\documentclass[11pt,letterpaper]{amsart}
\usepackage[english]{babel}
\usepackage[applemac]{inputenc}
\usepackage[T1]{fontenc}
\usepackage{bbm}
\usepackage{float, verbatim}
\usepackage{amsmath}
\usepackage{amssymb}
\usepackage{amsthm}
\usepackage{amsfonts}
\usepackage{graphicx}
\usepackage{mathtools}
\usepackage{enumitem}
\usepackage[pagebackref]{hyperref}
\usepackage{color}
\usepackage{tikz}
\usepackage{xcolor}
\usepackage{pgfplots}
\usepackage{caption}
\usepackage{subcaption}
\usepackage{enumitem}
\usepackage{url}
\usepackage{dsfont}
\usepackage{enumitem}
\usepackage{mathabx}

\allowdisplaybreaks 

\hypersetup{
	colorlinks=true,
	linkcolor=blue,
	filecolor=magenta,      
	urlcolor=cyan,
	citecolor= red,
}

\usepackage[colorinlistoftodos]{todonotes}

\pgfplotsset{compat=1.18}

\newcommand{\Haus}{\dim_{\mathrm{H}}}

\newtheorem*{thm*}{Theorem}
\newtheorem*{conj*}{Conjecture}
\newtheorem*{ques*}{Question}
\newtheorem*{rem*}{Remark}
\newtheorem*{defn*}{Definition}
\newtheorem*{mainques*}{Main questions}

\newtheorem{thm}{Theorem}[section]

\newtheorem{lma}[thm]{Lemma}
\newtheorem{cor}[thm]{Corollary}
\newtheorem{defn}[thm]{Definition}

\newtheorem{prop}[thm]{Proposition}
\newtheorem{conj}[thm]{Conjecture}
\newtheorem{claim}[thm]{Claim}
\newtheorem{rem}[thm]{Remark}
\newtheorem{ques}[thm]{Question}

\newtheorem{exm}[thm]{Example}

\def\R{\mathbb{R}}
\def\RR{\mathbb{R}}
\def\CC{\mathbb{C}}

\def\ZZ{\mathbb{Z}}
\def\NN{\mathbb{N}}

\def\supp{\mathrm{supp}}

\def \bxi {{\boldsymbol{\xi}}}
\def \bzero {{\mathbf{0}}}

\def \bN {{\mathbb N}}

\def \bR {{\mathbb R}}

\def \epsilon {{\varepsilon}}

\def \bt {{\mathbf{t}}}

\def \bzero {{\boldsymbol{0}}}

\def \bN {\mathbb N}
\def \bP {\mathbb P}

\def \bR {\mathbb R}

\def \ba {\mathbf a}

\def \bt {\mathbf t}

\def \bv {\mathbf v}
\def \bx {\mathbf x}

\def \by {\mathbf y}

\def \bzero {\mathbf 0}

\def \bxi {{\boldsymbol{\xi}}}
\def \Beta {{\boldsymbol{\eta}}}

\def \cD {\mathcal D}

\def \leq {\leqslant}

\def \geq {\geqslant}

\def \dim {\mathrm{dim}}

\def \det {\mathrm{det}}

\def \vol {\mathrm{vol}}

\def \supp {{\mathrm{supp}}}

\def \ds1 {\mathds{1}}

\def \epsilon {{\varepsilon}}

\numberwithin{equation}{section}

\title[Quantitative Fourier decay]{Fourier transform of nonlinear images of self-similar measures: quantitative aspects}

\author{Amlan Banaji}
\address{Amlan Banaji, Department of Mathematics and Statistics, University of Jyv\"askyl\"a, Agora, Mattilanniemi 2, Jyv\"askyl\"a, 40100, Finland}
\curraddr{}
\email{banajimath@gmail.com}

\author{Han Yu}
\address{Han Yu, College of Mathematics and Statistics, Centre of Mathematics, Chongqing University, Chongqing, 401331, China}
\curraddr{}
\email{han.yu.2@cqu.edu.cn}

\thanks{}


\begin{document}

\makeatletter
\providecommand\@dotsep{5}
\makeatother

\begin{abstract}
   This paper relates to the Fourier decay properties of images of self-similar measures $\mu$ on $\mathbb{R}^k$ under nonlinear smooth maps $f \colon \mathbb{R}^k \to \mathbb{R}$. 
   For example, we prove that if the linear parts of the similarities defining $\mu$ commute and the graph of $f$ has non-vanishing Gaussian curvature and $\kappa_2 > k/2$, then the Fourier dimension of the image measure is at least $\frac{2(2\kappa_2 - k)}{4 + 2\kappa_* - k}$, where $\kappa_2$ is the lower correlation dimension of $\mu$ and $\kappa_*$ is the Assouad dimension of the support of $\mu$. Under some additional assumptions on $\mu$, we use recent breakthroughs in the fractal uncertainty principle to obtain further improvements for the decay exponents. 
   
     We give several applications to nonlinear arithmetic of self-similar sets $F$ in the line. 
   For example, we prove that if $\dim_{\mathrm H} F>(\sqrt{65}-5)/4=0.765\dots$ then the arithmetic product set $F\cdot F = \{xy:x,y\in F\}$ has positive Lebesgue measure, while if $\dim_{\mathrm H} F>(-3+\sqrt{41})/4=0.850\dots$ then $F\cdot F\cdot F$ has non-empty interior. One feature of the above results is that they do not require any separation conditions on the self-similar sets.
\end{abstract}

\maketitle

\tableofcontents

\section{Introduction}\label{sec: intro}

\subsection{Main result}
 In this paper, the main topic is Fourier decay of nonlinear images of self-similar measures, a topic pioneered by Kaufman~\cite{K82}. Recall that the Fourier transform of a Borel probability measure $\mu$ supported on $\RR^k$ is the bounded continuous function $\widehat{\mu} \colon \RR^k \to \CC$ given by 
\begin{equation}\label{e:ftdef}
    \widehat{\mu}(\bxi) = \int_{\RR^k} e^{-2 \pi i \langle \bxi,\bx \rangle} d \mu(\bx) 
\end{equation}
where $\langle \cdot,\cdot \rangle$ is the Euclidean inner product. 
The question of which measures have Fourier transform decaying to $0$ as $|\bxi| \to \infty$, and the speed of decay if so, has received a great deal of attention in the literature. 
If $|\widehat{\mu}(\bxi)| \to 0$ as $|\bxi| \to \infty$ then $\mu$ is called \emph{Rajchman}, and if there exist $C,\sigma > 0$ such that $|\widehat{\mu}(\bxi)| \leq C |\bxi|^{-\sigma}$ for all $\bxi \in \RR^k \setminus \{ \mathbf{0} \}$ then $\mu$ is said to have \emph{polynomial Fourier decay} (or \emph{power Fourier decay}). For such a measure $\mu$, one is often interested in providing quantitative lower bounds for its \emph{Fourier dimension}, which is defined by 
\[ \dim_{\mathrm F} \mu \coloneqq \sup\{ \sigma \geq 0 : \exists C>0 \mbox{ s.t. } \forall \bxi \in \mathbb{R}^k \setminus \{ \mathbf{0} \}, |\widehat{\mu}(\bxi)| \leq C |\bxi|^{-\sigma/2} \}. \]
In this paper, we seek to provide a quantitative theory regarding the Fourier decay for nonlinear images of self-similar measures under maps $\mathbb{R}^k\to\mathbb{R}$. \footnote{In the subsequent article~\cite{BYqualitative}, we provide a qualitative theory (i.e. without quantifying the decay exponent) regarding the Fourier decay for nonlinear images under general non-degenerate maps $\mathbb{R}^k\to\mathbb{R}^d$.}

We now introduce our main result. Denote by $\mu_f$ the pushforward of a measure $\mu$ by a function $f$, defined by $\mu_f(B) = \mu(f^{-1}(B))$ for Borel sets $B \subseteq \RR$. 
The Fourier transform of $\mu_f$ is denoted $\widehat{\mu_f}$. 
The other terminology and notation used in the statement of Theorem~\ref{thm: image fourier decay} will be described briefly after the statement of the theorem and then in detail in Section~\ref{sec: pre}. 

\begin{thm}\label{thm: image fourier decay}
Let $k\geq 1$ be an integer. Let $\mu$ be a non-expanding self-similar measure on $\mathbb{R}^k$ with $\kappa_2 > k/2$. 
Let $f\colon \mathbb{R}^k\to\mathbb{R}$ be a twice continuously differentiable ($C^2$) function such that the graph $\Gamma_f=\{(\bx,f(\bx)) : \bx\in\mathbb{R}^k \} \subset \RR^{k+1}$ has non-vanishing Gaussian curvature over the support of $\mu$. 
Then there is some explicit $\sigma > 0$ such that the pushforward measure $\mu_f$ under $f$ satisfies 
\[ 
|\widehat{\mu_f}(\xi)|\ll |\xi|^{-\sigma} \qquad \mbox{ for all } \xi \in \RR \setminus \{0\}. 
\]
More precisely, one can choose arbitrarily 
\[
0 < \sigma < \frac{2\kappa_2 - k}{4 + 2\kappa_* - k} .
\]
If $d_{\infty} + \kappa_1 > k$ then one can choose arbitrarily \[0 < \sigma < \frac{d_{\infty} + \kappa_1 - k}{2 - k + 2\kappa_* + \kappa_1 - d_{\infty}}. \]
\end{thm}

\subsubsection*{A brief note on terminologies}
\begin{itemize}
    \item Here, $\kappa_2$ is the lower correlation dimension of $\mu$, $d_{\infty}$ is the Frostman exponent of $\mu$, $\kappa_1$ is the $l^1$-dimension of $\mu$, and $\kappa_*$ is the Assouad dimension of $\supp(\mu)$. 
    We make no separation assumptions on the self-similar measure. 
    For a general self-similar measure, $\mu$, $k \geq \kappa_* \geq \kappa_2 \geq d_{\infty} \geq \kappa_1$ (all are positive unless $\mu$ is a single atom). If $\mu$ is $s$-Ahlfors--David-regular then $s = \kappa_* = \kappa_2 = d_{\infty} \geq \kappa_1$, and if $s>k/2$ then we emphasise that our bound on $\sigma$ depends only on $s$ and not on the multiplicative constants of Ahlfors-regularity. 

    \item A self-similar measure is non-expanding if its semigroup of orthogonal transformations satisfies certain growth restrictions (for example if they commute). All self-similar measures in $\mathbb{R}^k, k \in \{1,2\}$ are non-expanding. 
    It may be challenging to prove good quantitative Fourier decay estimates for pushforwards of expanding self-similar measures, though we obtain non-quantitative results in~\cite{BYqualitative}. 
    
    \item We use Vinogradov and Bachmann--Landau notation: given complex-valued functions $f,g$ we write $f \ll g$ or $f = O(g)$ to mean $|f| \leq C|g|$ pointwise for some constant $C > 0$, and write $f \asymp g$ if $f \ll g$ and $g \ll f$. Subscripts may indicate parameters that the implicit constants are allowed to depend on. 
\end{itemize}

\subsection{Remarks about Theorem \ref{thm: image fourier decay}}\label{ss: remarksaboutmain}
We first describe how Theorem~\ref{thm: image fourier decay} relates to the existing literature. Sahlsten's survey~\cite{Sahlsten23survey} gives an overview of the topic of Fourier decay for fractal measures. 
For self-similar measures the problem is difficult, with a history going back to Erd\H{o}s~\cite{Erdos1,Erdos2}. 
There are many self-similar measures with polynomial Fourier decay, and many others which are not even Rajchman, but we will not go into details because in this paper we are concerned with \emph{nonlinear} fractal measures, where one often expects polynomial Fourier decay. 
Polynomial Fourier decay for pushforwards of self-similar measures by nonlinear maps $\RR \to \RR$ has been studied by Kaufman~\cite{K82}, Mosquera and Shmerkin~\cite{MS18}, and Algom, Chang, Wu and Wu~\cite{ACWW25}. 
Mosquera and Olivo \cite[Theorem~3.1]{MosqueraOlivo} consider pushforwards of homogeneous self-similar measures with non-trivial rotations under nonlinear holomorphic maps $\CC \to \CC$. 
Baker and Banaji~\cite{BB25} consider pushforwards of a class of measures which they call fibre product measures on $\RR^k$ under nonlinear maps $\RR^k \to \RR$. 
Theorem~\ref{thm: image fourier decay} builds on these results in the following ways. 
\begin{rem}
    \begin{itemize}
        \item For nonlinear pushforwards of self-similar measures on $\RR^k$ which do not satisfy the non-trivial fibre condition in~\cite{BB25}, even non-quantitative Fourier decay was not previously known until Theorem~\ref{thm: image fourier decay}. 
        \item When $k > 1$, Theorem~\ref{thm: image fourier decay} is the first \emph{quantitative} Fourier decay result for pushforwards of self-similar measures under maps $\RR^k \to \RR$. 
        \item In the very special case of homogeneous self-similar measures on the line, Mosquera and Shmerkin~\cite{MS18} obtained bounds on the exponent of decay. 
        As we will see in Example~\ref{exm: cantor lebesgue} below, Theorem~\ref{thm: image fourier decay} substantially improves these bounds. Nonetheless, our bounds are likely not optimal; some discussion is given in Section~\ref{ss: optimality}. 
        
    \end{itemize}
\end{rem}

A class of AD-regular (non-Rajchman) self-similar measures called missing digit measures arise from a generalised Cantor construction. 
The most famous example is the Cantor--Lebesgue measure (which is the self-similar measure on the middle-third Cantor set with weights $(1/2,1/2)$). For those measures, much is known about their $l^1$ and $l^2$-dimensions, which helps us to make the most of Theorem~\ref{thm: image fourier decay}. In those cases (and several others), we can use fractal uncertainty principles (FUP) to obtain $\varepsilon$-improvements over the bounds from Theorem~\ref{thm: image fourier decay}, see Section~\ref{ss:fup} for more details.

\begin{exm}[Cantor--Lebesgue measure]\label{exm: cantor lebesgue}
If $\mu$ is the Cantor--Lebesgue measure then it was shown in~\cite{CVY} that $\kappa_1 < 1/2$, so we use the $l^2$ bound. 
We have $\kappa_* = \kappa_2 = d_{\infty} = \log 2 / \log 3 > 1/2$ and $k=1$, so $\frac{2\kappa_2 - k}{4 + 2\kappa_* - k} = 0.061\dots$. 
In fact, we can use a FUP to show that there exists $\upsilon > 0$ (depending only on $\mu$ but not on $f$) such that the pushforward decays with exponent at least $0.061\dots + \upsilon$. 
For comparison, the bound from~\cite{MS18} is $0.016$. 
Ekstr\"om and Schmeling~\cite[page~86]{EkstromSchmeling} asked what the Fourier dimension of the image of the middle-third Cantor set by $x \mapsto x + x^2$ is; we obtain $0.12\ldots$ as a lower bound, which is partial progress. 
\end{exm}

\begin{exm}[Missing digit measures]
Fix an integer $b \geq 4$, divide the interval $[0,1]$ into pieces of size $1/b$ and choose $b-1$ of the pieces corresponding to a set $D \subset \{0,\dotsc,b-1\}$ (in the case $b=4$ we additionally assume that $D = \{0,1,2\}$ or $D = \{1,2,3\}$). Let $\mu$ be the self-similar measure with equal weights corresponding to the IFS of maps sending $[0,1]$ to each of the $b-1$ intervals. Then $\mu$ is non-Rajchman and $d_{\infty} = \kappa_2 = \kappa_* = \log(b-1)/\log b$. 

In this case the $l^2$ estimate proves that the exponent of Fourier decay for the image of $\mu$ under $x \mapsto x^2$ can be made close to $\frac{2 \log (b-1) - \log b}{2\log (b-1) + 3\log b}$. 
But Chow, Varj\'u and Yu~\cite[Proposition~2.4]{CVY} have recently shown that $\kappa_1 > 1/2$, so the $l^1$ estimate in fact gives that there exists $\epsilon_{b} > 0$ such that the true exponent is at least $\frac{2 \log (b-1) - \log b}{2\log (b-1) + 3\log b} + \epsilon_{b}$. Moreover, it is shown in \cite[Theorem~2.6]{CVY} that $\kappa_1 \to 1$ as $b \to \infty$, so a lower bound for the Fourier decay exponent is $1/3 - o_b(1)$. 
For comparison, if the Fourier transform of these measures decays at the fastest possible rate then the Fourier decay exponent would tend to $1/2$ as $b \to \infty$. 
\end{exm}

\subsection{Obstruction for Theorem \ref{thm: image fourier decay} for small self-similar measures}\label{sec: large linear space} 
If $\kappa_2 \leq k/2$ then Theorem~\ref{thm: image fourier decay} does not apply. 
This is not just due to our incapability. 
The key fact is that Theorem~\ref{thm: image fourier decay} allows the measure to be supported in a proper affine subspace of $\RR^k$ and thus some geometrical obstructions can occur. 
Indeed, for $k=2$, consider hyperboloids in $\mathbb{R}^3$ with negative curvatures. 
Some hyperboloids are `ruled surfaces': they contain (many) lines. For larger $k$, let $F\subset\mathbb{R}^{k+1}$ be a smooth hypersurface with non-vanishing Gaussian curvature. It is known that $F$ cannot contain any affine subspace with dimension bigger than $k/2$. On the other hand, some surfaces with non-vanishing Gaussian curvature contain large affine subspaces. Consider for example the following algebraic hypersurface with an even number $k>0$,
    \[
    F_k=\{x^3_1+\dots+x^3_{k+1}=1\}.
    \]
    Then the affine subspace \[ L_k:\{x_1=-x_2,x_3=-x_4,\dots,x_{k-1}=-x_{k}, x_{k+1}=1\}\] is contained in $F_k$. This subspace has dimension $k+1-(k/2)-1=k/2$. If $k$ is odd, then $F_k$ contains
    \[
    L_k:\{x_1=-x_2,x_3=-x_4,\dots,x_{k-2}=-x_{k-1}, x_{k}=2^{-1/3},x_{k+1}=2^{-1/3}\}
    \]
    which has dimension $k+1-(k-1)/2-2=(k-1)/2$. The Gaussian curvature of $F_k$ vanishes only if $x_1 x_2\dotsm x_{k+1}=0$. Therefore, $L_k$ contains some non-trivial open sets on which $F_k$ has non-vanishing Gaussian curvature. From here we see that there is a non-trivial open set $U$ in $\mathbb{R}^k$ so that $F_k\cap (U\times\mathbb{R})$ is the graph of some nonlinear function $f$. The projection $\pi(L_k)$ of $L_k$ to the first $k$-coordinates satisfies the property that $\pi(L_k)\cap U$ contains a non-trivial open set of $\pi(L_k)$ which is an affine subspace of dimension $\lfloor k/2 \rfloor$. 
    Consider any self-similar measure $\mu$ supported in $\pi(L_k)\cap U$. 
    We see that $f|_{\supp(\mu)}$ is in fact affine, although $f$ is itself a nonlinear function. Thus in general there is no hope to prove polynomial Fourier decay for $\mu_f$, which is an affine copy of $\mu$ (which may not even be Rajchman). 

    Since the obstruction described above does not apply when $k$ is odd and $\mu$ is $k/2$-dimensional, it is natural to ask whether one can expect polynomial Fourier decay for nonlinear pushforwards of such measures. Indeed, in Section~\ref{ss:fup} we will use a FUP due to Cladek and T.~Tao~\cite{CladekTaoFUP} that applies when the ambient spatial dimension is odd\footnote{If $k$ were even then there is no non-trivial FUP for general $k/2$-AD-regular sets, see \cite[Example~6.1]{DyatlovFUPsurvey}.} to prove the following. 
    \begin{thm}\label{thm: cladektao}
    Let $k \geq 1$ be an odd integer and let $\mu$ be a $k/2$-AD-regular non-expanding self-similar measure on $\RR^k$. 
    Let $f \colon \RR^k \to \RR$ be a $C^2$ map whose graph has non-vanishing Gaussian curvature over $\supp(\mu)$. Then $\mu_f$ has polynomial Fourier decay. 
    \end{thm}
    
    If $\mu$ is assumed not to be supported in any proper affine subspace, then there is no obstruction even if $\mu$ has small dimension (and regardless of the parity of $k$). 
    We use a FUP recently proved by Backus, Leng and Z.~Tao~\cite{BLTfup}\footnote{The $k=1$ case was due to Dyatlov and Jin~\cite{DJfupDolgopyat}.} to prove the following, which is a special case of the more general Conjecture~\ref{conj: dream conjecture} below. 
    \begin{thm}\label{thm: blt}
    Let $0<s\leq k$ and let $\mu$ be a non-expanding $s$-AD-regular self-similar measure on $\RR^k$ which is not supported in any proper affine subspace of $\RR^k$. 
    Let $f \colon \RR^k \to \RR$ be a $C^2$ map whose graph has non-vanishing Gaussian curvature over $\supp(\mu)$. 
    Then $\mu_f$ has polynomial Fourier decay. 
    \end{thm}

\subsection{More general pushforward maps}
Although the main focus of this paper is Fourier decay of images of self-similar measures under maps $\mathbb{R}^k\to\mathbb{R}$, the methods are viable for proving quantitative Fourier decay results in special cases regarding maps $\mathbb{R}^k\to\mathbb{R}^d$ for general $d\geq 1$, such as quadratic or holomorphic maps. 
In particular, for holomorphic functions, we can prove an analogue of van der Corput's lemma, which is a fundamental result in harmonic analysis. We first state a case of the classical result for comparison. 
\begin{lma}[van der Corput's lemma \cite{CarberyWrightVDC}]
    Let $l \geq 2$ be an integer and let $f \colon (0,1) \to \RR$ be smooth with $f^{(l)}(x) \geq 1$ for all $x \in (0,1)$. 
    Then 
    \[
    \Big| \int_0^1 e^{-2 \pi i \xi f(x)} dx \Big| \ll |\xi|^{-1/l}. 
    \]
\end{lma}

\begin{thm}\label{thm: complexvdc}
    Let $\mu$ be a self-similar measure on $\CC$ (which we regard as $\RR^2$) with $\kappa_2 > 1$, and let $U$ be an open neighbourhood of $\supp(\mu)$. 
    Let $l \geq 2$ be an integer and let $f \colon U \to \CC$ be holomorphic with $f^{(l)}(z) \neq 0$ for all $z \in \supp(\mu)$. 
    Then for all $\varepsilon > 0$, for all $\bxi \in \RR^2 \setminus \{0\}$, 
    \[ 
    |\widehat{\mu_f}(\bxi)| = \Big| \int_{\CC} e^{-2 \pi i \langle \bxi,f(z)\rangle} d\mu(z) \Big| \ll |\bxi|^{-\frac{\kappa_2 - 1}{1 + \kappa_* + (l-2) \kappa_* / d_{\infty}} + \varepsilon}.
    \]
\end{thm}

There are versions of van der Corput's lemma where the underlying measure is absolutely continuous on $\RR^k$~\cite{CarberyWrightVDC}, or a fractal measure in the line~\cite{ACWW25}, but we are not aware of versions for fractal measures in the plane prior to Theorem~\ref{thm: complexvdc}. 

\subsection{Applications}
Theorem~\ref{thm: image fourier decay} can be used to prove many results about nonlinear arithmetic of self-similar sets. 
One special case of what we prove in Section~\ref{sec: arithmetic} is the following, which immediately implies the bounds stated in the abstract. 
\begin{thm}\label{thm: arith headline}
Let $E,F,G$ be self-similar sets in $\RR$. 
\begin{itemize}
    \item If 
    \[ \min\{\Haus E,\Haus F\}>\frac{\sqrt{65}-5}{4} = 0.765\dotsc \]
    then $E \cdot F \coloneqq \{xy : x \in E, y \in F\}$ has positive Lebesgue measure. 
    \item
    If 
    \[ \min\{\Haus E,\Haus F,\Haus G\}> \frac{-3+\sqrt{41}}{4} = 0.850\dotsc \] 
    then $E\cdot F\cdot G$ has non-empty interior. 
    \end{itemize}
\end{thm}
We will motivate these results using projection theory in Section~\ref{sec: arithmetic}. 
In brief, while there is a good understanding of the dimension thresholds which guarantee that arithmetic products of self-similar sets are large in the sense of dimension, it is considered much more challenging to find conditions (such as those from Theorem~\ref{thm: arith headline}) which guarantee that arithmetic products are large in the sense of Lebesgue measure or non-empty interior.

Knowing good bounds on the quantitative Fourier decay of a measure on the line has numerous other benefits, see for example \cite{Sahlsten23survey} and \cite[Section~2]{BB25}. 
We will not describe these applications in detail because they are essentially immediate consequences of existing results in the literature; instead, we point the reader to the results which can be applied.  
\begin{enumerate}
    \item In the setting of Theorem~\ref{thm: image fourier decay}, for $\mu$-typical $\bx$, one can deduce that $f(\bx)$ is a normal number using a result of Davenport, Erd{\H{o}s} and LeVeque~\cite{DEL}, and obtain a good rate of equidistribution of sequences such as $b^n f(\bx) \mod 1$ by Pollington et al.~\cite[Theorems~1--3]{PVZZ}. 
    \item Again in the setting of Theorem~\ref{thm: image fourier decay}, in the context of the Fourier uniqueness problem, $f(\supp(\mu))$ is a set of multiplicity. 
    This follows from an old result of Salem~\cite{Salem}, since $f(\supp(\mu))$ supports a Rajchman measure. 
    \item One can deduce Fourier restriction estimates for many pushforward fractal measures for a range of exponents which depends on the exponent of Fourier decay and the Frostman exponent of the pushforward measure, by Mitsis \cite[Corollary~3.1]{Mitsis} and Mockenhaupt \cite[Theorem~4.1]{Mockenhaupt}. 
    To give one concrete example, fix $\mu$ to be the Cantor--Lebesgue measure on $[0,1]$. 
    Then for all $C^2$ maps $f \colon [0,1] \to \RR$ with positive first and second derivatives, $\mu_f$ has Frostman exponent $s = \log 2 / \log 3$. 
    Therefore by \cite[Theorem~4.1]{Mockenhaupt} and Example~\ref{exm: cantor lebesgue}, there is $\varepsilon > 0$ independent of $f$ such that for all 
    \[
    1 \leq p \leq \frac{2(2 - 2s + 2\frac{2s-1}{4+2s-1})}{4(1-s) + 2\frac{2s-1}{4+2s-1}} + \varepsilon = 1.076 \dotsc + \varepsilon 
    \]
    we have 
    \[
    \left( \int|\widehat{\varphi}(f(\xi))|^2 d\mu(\xi) \right)^{1/2} \ll_{f,p} ||\varphi||_{L^p(\RR)}
    \]
    for all Schwartz functions $\varphi \colon \RR \to \CC$. 
\end{enumerate}

\subsection{Structure of paper}
In Section~\ref{sec: pre}, we formally introduce the notions and terminology (especially from fractal geometry) that is used in the statement of our main result and needed for the proof. 

Section~\ref{sec: provequantitative} proves our main result Theorem~\ref{thm: image fourier decay} as a consequence of Theorem~\ref{thm: allparams}, which involves the more general Fourier $l^p$ dimensions. 
The key ideas are already present in the case where the linear parts of the contractions defining $\mu$ are equal (Lemma~\ref{lma: homogquant}). 
We begin by relating the Fourier transform of a pushforward self-similar measure to the Fourier transform of a measure lifted onto the graph, decomposing the latter into small pieces, and linearising those small pieces (Lemmas~\ref{lma: lifttograph} and~\ref{lma: startingdecomp}). 
The idea is then to obtain a bound in terms of an average of magnitudes of the Fourier transform of the original self-similar measure across a range of frequencies. These frequencies are well separated because of the non-vanishing Gaussian curvature assumption, so the average can be bounded using H\"older's inequality, giving a term involving a integral of $|\widehat{\mu}|^p$ which can be bounded using the $l^p$ dimensions. 
We anticipate that this general strategy of obtaining a bound in terms of an integral of $|\widehat{\mu}|^p$ could be useful to prove Fourier decay properties of different classes of fractal measures, for example certain self-similar measures. 

Section~\ref{sec: arithmetic} gives applications related to nonlinear arithmetic, in particular proving Theorem~\ref{thm: arith headline} and interpreting the results in terms of exceptional directions for nonlinear projections. 
In fact, we prove analogous results for self-similar measures and deduce the results for sets by noting that self-similar sets support AD-regular self-similar measures with an arbitrarily small loss in dimension. 
The results for measures are proved by transforming to logarithmic space so multiplicative convolution becomes additive convolution, and using Theorem~\ref{thm: image fourier decay} to get quantitative Fourier decay for the measures in logarithmic space. This Fourier decay can be used to show that the convolved measures have densities with the required distributional or regularity properties. 

Section~\ref{sec: further dev} begins by describing further Fourier decay results that can be obtained in addition to Theorem~\ref{thm: image fourier decay}. 
We use various fractal uncertainty principles to get small improvements for some of our Fourier decay bounds, and prove Theorems~\ref{thm: cladektao} and~\ref{thm: blt}. 
We then prove quantitative Fourier decay estimates for pushforwards of self-similar measures by quadratic maps $\RR^k \to \RR^d$, and the van der Corput-type result Theorem~\ref{thm: complexvdc}. 
These proofs use nice properties of the determinant of the Hessian of $\bx \mapsto \sum_i v_i f_i(\bx)$ that hold for these classes of maps. 
We also pose questions about how much the Fourier decay and nonlinear arithmetic results can be improved and what one might expect the optimal bounds to be. 

\section{Preliminaries}\label{sec: pre}
A general account of fractal geometry can be found in textbooks such as \cite{Fa,Ma2}. 

\subsection{Dimensions of sets}
\subsubsection*{Hausdorff dimension}
Let $k\geq 1$ be an integer. Let $E\subset\mathbb{R}^k$ be a non-empty Borel set. Let $g\colon [0,1)\to [0,\infty)$ be a continuous function such that $g(0)=0$. Then for all $0<\delta\leq 1$ we define the  quantity
\[
\mathcal{H}^g_\delta(E)=\inf\left\{\sum_{i=1}^{\infty}g(\mathrm{diam} (U_i)): \bigcup_i U_i\supset E, \mathrm{diam}(U_i)<\delta\right\}.
\]
The $g$-Hausdorff measure of $E$ is
\[
\mathcal{H}^g(E)=\lim_{\delta\to 0} \mathcal{H}^g_{\delta}(E).
\]
When $g(x)=x^s$ then $\mathcal{H}^g=\mathcal{H}^s$ is the $s$-Hausdorff measure, and the Hausdorff dimension of $E$ is
\[
\Haus E=\inf\{s\geq 0:\mathcal{H}^s(E)=0\}=\sup\{s\geq 0: \mathcal{H}^s(E)=\infty \}.
\]

\subsubsection*{Box dimension}
If we assume $E$ is bounded then for all $\delta > 0$ let $N_\delta(E)$ be the smallest number of open balls of diameter $\delta$ needed to cover $E$. 
In this case the lower and upper box (or Minkowski) dimensions of $E$ are defined by 
\[ \underline{\dim}_{\mathrm B} E = \liminf_{\delta \to 0^+} \frac{\log N_{\delta}(E)}{\log(1/\delta)}; \qquad \overline{\dim}_{\mathrm B} E = \limsup_{\delta \to 0^+} \frac{\log N_{\delta}(E)}{\log(1/\delta)}. \]

\subsubsection*{Assouad dimension}
The Assouad dimension $\dim_{\mathrm A} E$ (also known as Furstenberg's $*$-exponent, denoted $\kappa_*$, see~\cite{Fu2008}), of a non-empty set $E \subset \RR^k$ is defined by 
\begin{align*}
\dim_{\mathrm A} E \coloneqq \inf\Big\{ \alpha \geq 0 : &\exists C>0 \mbox{ such that for all } 0 < r < R \mbox{ and } \bx \in E, \\ &N_r(B_R(\bx) \cap E) \leq C \left(\frac{R}{r}\right)^{\alpha} \Big\}.
\end{align*}
For more on the Assouad dimension we refer the reader to~\cite{Assouad,Fraser2020book}. 
For all non-empty bounded $E \subset \RR^k$ we have 
\[ \Haus E\leq \underline{\dim}_{\mathrm B} E \leq \overline{\dim}_{\mathrm B} E \leq \dim_{\mathrm A} E. \]

\subsection{Dimension exponents of measures}
\subsubsection*{Dyadic cubes, $L^q$ dimension, and Frostman exponent}
For each integer $n\geq 1$, let $\cD_n$ be the decomposition of $\mathbb{R}^k$ into the disjoint union of dyadic cubes of sidelength $1/2^n$ given by translates of $2^{-n}[0,1)^k$ by $2^{-n}\ZZ^k \coloneqq \{2^{-n}\bx : \bx \in \ZZ^k\}$. 
Given $\bx \in \RR^k$, denote by $\cD_n(\bx)$ the unique element of $\cD_n$ containing $\bx$. 
We align all the cubes so that $\mathbf{0}$ is the corner of at least one (thus $2^k$) many such dyadic cubes.

Let $\mu$ be a Borel probability measure on $\mathbb{R}^k$. 
For $q \in (1,\infty)$, the $L^q$ dimension of $\mu$ is 
\[ 
d_q = d_q(\mu) \coloneqq \liminf_{n \to \infty} \frac{- \log \sum_{D \in \mathcal{D}_n} (\mu(D))^q}{n(q-1)\log 2}. 
\]
It is well known that $d_q$ is continuous and decreasing on $(1,\infty)$, see \cite[Lemma~1.7]{Sh}. 

The (uniform) Frostman exponent of $\mu$, denoted $d_{\infty}$ (or $d_{\infty}(\mu)$), is the supremum of $\kappa \geq 0$ such that
\[
\mu(B_r(\bx))\ll r^\kappa
\]
uniformly for all $\bx\in\mathbb{R}^k$, $r>0$. 

\subsubsection*{AD-regularity}
Let $\mu$ be a Borel probability measure on $\mathbb{R}^k$. 
If there exist $s \geq 0$ and $C \geq 1$ such that 
\[
C^{-1} r^s \leq \mu(B_r(\bx)) \leq C r^s
\]
uniformly for all $\bx\in\supp(\mu)$, $r\in (0,1)$, then $\mu$ and $\supp(\mu)$ are called \emph{$s$-Ahlfors--David regular} (with constant $C$), or simply \emph{AD-regular} if the value of $s$ is clear from the context. 
We call $s$ the (uniform) AD exponent of $\mu$. 

\subsection{Fourier $l^p$ dimension of measures}

The Fourier $l^p$ dimension of a Borel probability measure $\mu$, where $p > 0$, is defined by 
\begin{equation}\label{e:lpdef}
\kappa_p = \kappa_p(\mu) \coloneqq \sup\Big\{ 0 \leq s < k : \int_{B_R(\mathbf{0})} |\widehat{\mu}(\bxi)|^p d \bxi \ll R^{k-s} \Big\}.
\end{equation}
The values $p=1,2$ are especially important to us. 
The $l^1$-dimension was studied in detail in \cite{CVY,Yu}. 
For compactly supported Borel probability measures, the better-known $l^2$-dimension is also called the (lower) correlation dimension and has several equivalent definitions. 
\begin{lma}\label{lma: correlationdim}
    Let $\mu$ be a compactly supported Borel probability measure on $\RR^k$. Then, with the convention $\sup \varnothing = 0$, 
    \[
    d_2 = \kappa_2 = \sup \left\{ 0 \leq s < k : \int_{\RR^k} |\bxi|^{s-k} |\widehat{\mu}(\bxi)|^2 d\bxi < \infty \right\} \leq \Haus (\supp(\mu)). 
    \]
\end{lma}
\begin{proof}
    A proof of the first equality is given in \cite[Lemma~2.5]{FNW} in $\RR$ and \cite[Corollary~4.4]{FalconerZhang} in higher dimensions. 
    The second equality follows from the formula for energy in terms of Fourier transform, see \cite[Theorem~3.10]{Ma2}. 
    The final inequality is \cite[Lemma~2.5]{FNW}. 
\end{proof}

We record the following inequalities related to the Fourier $l^p$ dimensions. 
\begin{lma}
    If $\mu$ is a Borel probability measure with compact support in $\RR^k$ and $0 < p \leq q$, then 
    \[ \frac{p}{q} \cdot \kappa_q \leq \kappa_p \leq \kappa_q .\]
\end{lma}
\begin{proof}
The second inequality holds since $\mu$ is a probability measure, so $|\widehat{\mu}(\bxi)| \leq 1$ and hence $|\widehat{\mu}(\bxi)|^p \geq |\widehat{\mu}(\bxi)|^q$ for all $\bxi$. 
For the first inequality, we apply H{\"o}lder's inequality with $p' \coloneqq q/p$ and $q'$ such that $1/p' + 1/q' = 1$ to the function $|\widehat{\mu}(\bxi)|^p$ and the indicator function of $B_R(\mathbf{0})$. This gives that for all $\varepsilon > 0$, 
\[  \int_{|\bxi| \leq R} |\widehat{\mu}(\bxi)|^p d\bxi \ll R^{k/q'} \left( \int_{|\bxi| \leq R} |\widehat{\mu}(\bxi)|^{pp'} d\bxi \right)^{1/p'} \ll R^{k/q' + (k-\kappa_{pp'})/p' + \varepsilon}. \]
In particular, 
\[ \kappa_p \geq k - \frac{k-\kappa_q}{p'}  - \frac{k}{q'} - \varepsilon = \frac{p}{q} \cdot \kappa_q - \varepsilon. \]
As $\varepsilon$ was arbitrary this completes the proof. 
\end{proof}

We will use the following facts. 
\begin{lma}\label{lma: l2dimHausdorff}
If $\mu$ is an $s$-AD-regular Borel probability measure then 
\[ \kappa_2 = s.\]
\end{lma} 
\begin{proof}
This follows from \cite[Section~3.8]{Ma2}. 
\end{proof}

\begin{lma}
    If $\mu$ is a compactly supported Borel probability measure on $\RR^k$ then $\kappa_1 \leq d_{\infty}$. 
\end{lma}
\begin{proof}
    Fix $\varepsilon \in (0,1)$ and let $\phi$ be a non-negative smooth function with compact support inside $B_2(\mathbf{0})$ taking value $1$ on $B_1(\mathbf{0})$. For $r>0$ write $\phi_r(\bx) = \phi(\bx/r)$. 
    Now for all $\by \in \RR^k$ and $r \in (0,1)$, using Plancherel's theorem, 
    \begin{align*}
    \mu(B_r(\by)) &\leq \int \phi_r(\bx - \by) d\mu(\bx) \leq \int |\widehat{\phi_r}(\bxi)| |\widehat{\mu}(\bxi)| d\bxi \\ 
    &\ll \int_{|\bxi| \leq r^{-(1+\varepsilon)}}|\widehat{\phi_r}(\bxi)| |\widehat{\mu}(\bxi)| d\bxi + \int_{|\bxi| > r^{-(1+\varepsilon)}}|\widehat{\phi_r}(\bxi)| d\bxi \\
    &\ll r^k r^{-(1+\varepsilon)(k-\kappa_1 + \varepsilon)} + r^{k}.
    \end{align*}
    Therefore $d_{\infty} \geq \kappa_1 - \varepsilon - \varepsilon k + \varepsilon \kappa_1 - \varepsilon^2$. 
    Since $\varepsilon$ was arbitrary this completes the proof. 
\end{proof}

Using Plancherel's theorem, it is possible to see that for any compactly supported Borel probability measure on $\mathbb{R}^k$ and diffeomorphism $f$ on $\mathbb{R}^k$, the pushforward $\mu_f$ and $\mu$ share the same $\kappa_2$ value. 
In the set setting, this result reflects the fact that Hausdorff dimension is kept invariant under diffeomorphisms. 
It is natural to ask the following. 
\begin{ques}
What happens to the other $l^p$ dimensions under diffeomorphisms? 
In particular, what can one say about $l^1$ dimension under diffeomorphisms?
\end{ques}

\subsection{Self-similar sets and measures}\label{ss:ifs}
    Let $k\geq 1$ and $N>1$ be integers, and let $D \subset \RR^k$ be non-empty and compact. 
    Let $\Lambda = \{f_i \colon D \to D\}_{1 \leq i \leq N}$ be contraction maps (i.e. $\rho$-Lipschitz maps for some $\rho < 1$); $\Lambda$ is called an \emph{iterated function system} or IFS for short. 
    Let $p_1,\dots,p_N\in (0,1)$ be such that $\sum_i p_i=1$. By Hutchinson's theorem~\cite{Hutchinson}, there is a unique non-empty compact set $K$, called the \emph{attractor}, and a unique Borel probability measure $\mu$ supported on $K$, such that
    \begin{equation*}
    K=\bigcup_{i}f_i(K), \qquad \mu=\sum_{i} p_i f_i(\mu). 
    \end{equation*}
    We will always assume that the contractions do not share a common fixed point. 

    If there exist $r_1,\dots,r_N\in (0,1)$ and orthogonal transformations $O_1,\dotsc,O_N\in O_k(\mathbb{R})$ (possibly including reflections) and $\bt_1,\dotsc,\bt_N\in\mathbb{R}^k$ such that $\Lambda$ consists of the similarity maps
    \[
    \Lambda=\{f_i(\cdot)=r_iO_i(\cdot)+\bt_i\}_{i\in\{1,\dots,N\}},
    \]
    then we say that $\Lambda, K, \mu$ are self-similar. 
    In this case, we say that $\Lambda,\mu,K$ are homogeneous if $r_iO_i$ are all the same for $i\in\{1,\dots,N\}$. 

\subsection{Dimensions of self-similar sets and measures}\label{ss:dims}
The dimension of self-similar sets and measures is an important topic in geometric measure theory. Here, we review some of the important results. 
\begin{itemize}
    \item The box dimension of any self-similar set exists and coincides with Hausdorff dimension. This follows from Falconer's so-called `implicit theorems,' see \cite[Example~2]{KennethImplicit}.    
    \item Given a self-similar measure $\mu$ as above, $\kappa_{sim} \coloneqq \frac{\sum_i p_i \log p_i}{\sum_i p_i \log r_i}$ is called the similarity dimension, and it is always an upper bound for Hausdorff dimension of $\mu$. 
    The unique $s \geq 0$ satisfying $\sum_{i} r_i^s = 1$ is called the similarity dimension of the IFS (or of the self-similar set). It is always an upper bound for Hausdorff dimension of the set. 
    
    \item Given an IFS, one often assumes that certain separation conditions. 
    We say that $\Lambda, K, \mu$ has the strong separation condition (SSC) if $\forall i\neq j$, $f_i(K)\cap f_j(K)=\varnothing$. 
    
    A weaker condition is the open set condition (OSC), which means there exists a non-empty bounded open set $V$ such that $V \supseteq \bigcup_{i} f_i(V)$ with the union disjoint. 
    
    If the IFS consists of similarity maps on $\RR$, then another condition that is even weaker than the OSC is the exponential separation condition (ESC), introduced in~\cite{H14}. 
    In this case writing $r_i = r_{i_1}\dotsb r_{i_n}$ for $i \in \{1,\dotsc,N\}^n$, define the distance $d(i,j)$ between cylinders $i,j \in \{1,\dotsc,N\}^n$ to be $\infty$ if $r_i O_i \neq r_j O_j$, and $|f_i(0) - f_j(0)|$ if $r_i O_i = r_j O_j$. 
    Then the ESC holds if there exists $c>0$ such that for infinitely many $n$, $d(i,j) > c^n$ for all distinct $i,j \in \{1,\dotsc,N\}^n$. 
    The ESC holds extremely generically: for many families of self-similar IFSs parameterised in a real-analytic fashion, the set of `exceptional' parameters where it fails has zero Hausdorff dimension, see \cite[Theorem~1.8]{H14}. 

    In $\mathbb{R}$, under any of these separation conditions, the dimension of the self-similar set and measure coincide with the minimum of $1$ and the respective similarity dimension~\cite{Fa,H14}. 
    \item  
    Given a self-similar IFS on $\RR$ satisfying the OSC, there exists a unique self-similar measure for which $\kappa_2$ equals the Hausdorff dimension $H$ of $\supp(\mu)$. 
    It is called the measure of maximal dimension and corresponds to weights $r_i^H$. Moreover, $d_{\infty}=\kappa_2= H = \kappa_{sim}$~\cite{H14,Sh}. Under the OSC, this measure is $\kappa_2$-AD-regular. 

    If $\mu$ is an arbitrary self-similar measure on $\RR$ (i.e. not necessarily the measure of maximal dimension), one has $\kappa_2 \leq \Haus \mu \leq \Haus (\supp (\mu) )$. 
    Shmerkin~\cite[Theorem~6.6]{Sh} proved that under the ESC we have 
    \[ d_q = \min\Big\{ \frac{T(\mu,q)}{q-1}, 1 \Big\}\] 
    where for $q>1$ one defines $T(\mu,q)$ to be the unique solution to $\sum_i p_i^q r_i^{-T(\mu,q)} = 1$. 

    \item Dimension exponents and $L^q$ dimensions of overlapping self-similar measures in $\RR^k$ for $k \geq 1$ are studied in~\cite{HochmanHigher,CorsoShmerkinHigher}. 

    \item The Assouad dimension of self-similar sets (which appears in Theorem~\ref{thm: image fourier decay}) is well understood, see~\cite{FraserAssouadSelfSim,GarciaAssouad} and \cite[Chapter~7]{Fraser2020book}. 

    \item Self-similar measures which are not a single atom are known to have a positive Frostman exponent, see the proof of \cite[Proposition~2.2]{FL}, for example. 
\end{itemize}

\subsection{Non-expanding generators}
Finitely generated sub-semigroups of linear groups enjoy the following phenomena known as the Tits alternative.
\begin{thm}[\cite{OS95,Tits}]\label{thm: tits}
    Let $k \in \NN$ and $G$ be a finitely generated sub-semigroup of $GL(k,\mathbb{R})$. Let $F$ be a finite generating set. Then the following are equivalent: 
    \begin{itemize}
        \item $G$ (or equivalently the group generated by $G$) is virtually nilpotent 
        \item $G$ does not contain a non-abelian free sub-semigroup
        \item $G$ satisfies a non-trivial semigroup identity
        \item $G$ has sub-exponential growth (i.e. is non-expanding) with respect to $F$. 
    \end{itemize}
\end{thm}
We will not define all the algebraic terminology because the only property we will use in the proofs is the last one.
\begin{defn}
Let $G$ be a group. Let $F\subset G$ be a finite set. We say that $G$ is \emph{non-expanding} with respect to $F$ (or $F$ is \emph{non-expanding}) if for each $\epsilon>0$,
\[
\# F_n\ll e^{\epsilon n},
\]
where $F_n\subset G$ is the collection of $f_1 \dotsb f_n$ for $f_1,\dots,f_n\in F$. 
\end{defn}

It is not difficult to see that whether a finitely generated (semi)group is non-expanding does not depend on the choice of generating set. The Tits alternative gives a precise characterisation of which subgroups and sub-semigroups of $GL(k,\RR)$ are non-expanding.

\begin{defn}[non-expanding self-similar system]
Let $\Lambda$ be a self-similar IFS. We say that it is \emph{non-expanding} if the collection of orthogonal parts $\{O_i\}$ is non-expanding when viewed as a subset of the Euclidean group on $\mathbb{R}^k$. 
We also say that a self-similar set/measure is non-expanding if one of its corresponding self-similar IFSs is non-expanding.
\end{defn}
If a self-similar system is homogeneous (all linear parts are equal) then it is non-expanding. More generally, it is non-expanding if the group generated by the orthogonal parts is finite (even if it is non-abelian, for instance the group of rotations by multiples of $\pi/2$ about the coordinate axes in $\RR^3$). 
More generally still, virtually abelian groups (i.e. groups with finite index abelian subgroups) are non-expanding, and this property automatically holds for self-similar systems on $\RR$ or $\RR^2$. 
For $k \geq 3$, a self-similar system on $\RR^k$ may or may not be non-expanding. 
\begin{exm}
    Let $k=3$ and consider the collection $\{r_1,r_2\}\subset SO(3,\bR)$, where $r_1,r_2$ are rotations around the $x$ and $y$ axes respectively, with the same rotation angle $\arccos(1/3)$. 
    It is well-known (c.f. Banach--Tarski paradox) that $r_1,r_2$ generate a free group, hence the subgroup of $SO(3,\bR)$ generated by $r_1,r_2$ is expanding, and an IFS such as 
    \[ \{\bx \mapsto r_1(\bx)/2,\bx \mapsto r_2(\bx)/2 + (0,0,1)\} \] 
    is expanding. 
\end{exm}

\subsection{Gaussian curvature}
Let $M\subset\mathbb{R}^k$ be a smooth hypersurface. The notion of Gaussian curvature is defined on $M$ as a function of $m\in M$. 
For each $m\in M$, we denote the unit normal vector $N_m\in \mathbb{S}^{k-1}$. 
We need to fix an orientation beforehand and we will assume this. Then we can define a smooth map $G\colon m\in M\mapsto N_m\in \mathbb{S}^{k-1}$. This map is only well-defined locally unless $M$ is orientable in which case it is defined globally. 
The shape operator at $m\in M$ is defined to be $dG$ at $m$. 
It is a real symmetric operator and therefore it has real eigenvalues and eigenvectors. Its eigenvalues are principal curvatures, its trace is the mean curvature and its determinant is the Gaussian curvature. 

If $M$ is the graph of a $C^2$ function $f$, then it is well known that at each point $(\bx,f(\bx))\in M$, the Gaussian curvature is proportional to the determinant of the Hessian matrix of $f$ at $\bx$. 
The proportion is a continuous function and it is nowhere zero.

\section{Proof of Theorem \ref{thm: image fourier decay}}\label{sec: provequantitative}

\subsection{Average of Fourier transform}

In the proof of Theorem~\ref{thm: image fourier decay} we will need to bound the Fourier transform of a measure at a frequency by an integral of its Fourier transform. 

\begin{lma}\label{lma: compactly supported}
    Let $\mu$ be a compactly supported Borel probability measure on $\mathbb{R}^k$ and fix a large $L > 0$. Then for all $\ba \in \RR^k$, 
    \begin{equation*}
        |\widehat{\mu}(\ba)| \ll \int_{B_{R}(\ba)} |\widehat{\mu}(\bxi)| d\bxi + R^{-L}
    \end{equation*}
    as $R \to \infty$, with implicit constant independent of $R$ and $\ba$. 
\end{lma}

\begin{proof}
Without loss of generality, assume that $\mu$ is supported inside the unit ball $B_1(\bzero)$ centred at the origin. Let $\phi$ be a smooth bump function with value $1$ in $B_1(\bzero)$ and zero outside of $B_2(\bzero)$. Consider the measure $\phi\mu$ giving mass $\int_A \phi(\bx) d\mu(\bx)$ to a Borel set $A \subseteq \RR^k$. 
By construction, $\phi\mu=\mu$, so the convolution theorem gives that for each $\ba \in \mathbb{R}^k$, 
\[
\widehat{\mu}(\ba)=\widehat{\phi\mu}(\ba) = \int_{\RR^k}\widehat{\mu}(\bxi)\widehat{\phi}(\ba - \bxi)d\bxi.
\]
Now, since $\mu$ is a probability measure, we always have $|\widehat{\mu}(\bxi)| \leq 1$, and also $\max_{\bxi \in \RR^k} |\widehat{\phi}(\bxi)| < \infty$. 
Therefore 
\begin{equation*}
    |\widehat{\mu}(\ba)| \ll \int_{B_{R}(\ba)} |\widehat{\mu}(\bxi)| d\bxi + \int_{\RR^k \setminus B_{R}(\ba)} |\widehat{\phi}(\ba - \bxi)| d\bxi \ll \int_{B_{R}(\mathbf{a})} |\widehat{\mu}(\bxi)| d\bxi + R^{-L}, 
\end{equation*}
as required. 
\end{proof}

\subsection{Decomposing and linearising the Fourier transform}
We next observe that the Fourier transform of a pushforward of a measure is the same as the Fourier transform of the lift of the measure to the graph of the pushforward function evaluated at a corresponding frequency. 
In this paper we need only the $d=1$ case of Lemmas~\ref{lma: lifttograph} and~\ref{lma: startingdecomp}, but we prove them for general $d$ because this will be useful in~\cite{BYqualitative} and the proofs are very similar. 
Lemma~\ref{lma: lifttograph} will be applied for self-similar measures $\mu$, but we do not yet need to assume that $\mu$ is self-similar. 
\begin{lma}\label{lma: lifttograph}
Let $\mu$ be a compactly supported Borel probability measure on $\RR^k$ and let $f \colon \RR^k \to \RR^d$ be continuous. 
Consider the map $T\colon \RR^k \to \RR^{k+d}$, $\bx \mapsto (\bx,f(\bx))$ and the image measure $\mu_T$. 
Then for all $\bxi \in \RR^d$, letting $\mathbf{0}$ be the zero vector with $k$ entries,  
\[ \widehat{\mu_T}(\mathbf{0},\bxi) = \widehat{\mu_f}(\bxi).  \] 
\end{lma}
\begin{proof}
This is simply a matter of unpacking the definition of the Fourier transform from~\eqref{e:ftdef}. 
\end{proof} 

Now let $\mu$ be a self-similar measure on $\RR^k$. 
Without loss of generality, we can assume that the diameter of the support of $\mu$ is at most $1$. Let $\Lambda$ be an IFS for $\mu$. For each $\omega\in \Lambda^\bN$, following the composition of maps $\omega_0\omega_1\dotsb$, there is a least $l\geq 0$ such that the contraction ratio corresponding to $\omega_0^{l} \coloneqq \omega_0\dotsb \omega_l$ is at most $1/2^n$. We write $l=l_{\omega,n}$. Notice that the finite collection of paths
\[
\{\omega_{0}^{l_{\omega,n}}\}_{\omega\in\Lambda^{\mathbb{N}}}
\]
corresponds to a covering of $\supp(\mu)$. 
Such a covering uses similar copies of $\supp(\mu)$ of sizes at most $1/2^n$ and at least $\rho_m 2^{-n}$ where $\rho_m$ is the smallest contraction ratio for $\mu$. Let $\mu_{\omega,n}$ be the corresponding similar copy of $\mu$; we denote its total mass to be $|\mu_{\omega,n}|$. Clearly, each such branch $\supp(\mu_{\omega,n})$ intersects at least one and at most $2^k$ dyadic cubes in $\cD_n$. We choose only one such dyadic cube and associate it with $\mu_{\omega,n}$. As a result, for each dyadic cube $D_n\in\cD_n$, there is a collection of similar copies $\mu_{\omega,n}$ that are associated with this cube. We write $\mu_{\omega,n}\sim D_n$ for this association. We will call such a collection of $\omega$ to be $\Lambda_{D_n}$. We then obtain the following decomposition of $\Lambda^{\mathbb{N}}$:
\[
\{\Lambda_{D_n}\}_{D_n\in\cD_n}
\]
which induces a decomposition (or disintegration) of $\mu$. For each $\omega\in\Lambda_{D_n}$, the copy $\mu_{\omega,n}$ is supported in $3D_n$, the tripling of $D_n$ with the same centre. 

Given a $C^2$ map $f \colon \RR^k \to \RR^d$, write $\Gamma_f=\{ (\bx,f(\bx)):\bx\in \RR^k \} \subset \RR^{k+d}$ for the graph of $f$. 
Let $T_{\omega,n}$ be the ($k$-dimensional) tangent plane of $\Gamma_f$ at $\by_0=(\bx_0,f(\bx_0))$ for $\bx_0$ being the centre of $D_n$. 
Given $\mu_{\omega,n} \sim D_n \subset \RR^k$ and $\bxi \in \RR^{k+d}$ with first $k$ coordinates $0$, let $\bxi_{\omega,n} \in \RR^k$ be such that $(\bxi_{\omega,n} ,\bzero)$ is the projection of $\bxi$ to $\mathbb{R}^k\times\bzero \subset \RR^{k+d}$ along the direction that is orthogonal to $T_{\omega,n}$, i.e. 
\begin{equation}\label{e:definexiomegan}
((\bxi_{\omega,n},\bzero)-\bxi) \perp T_{\omega,n}.
\end{equation}
Note that in the following statement we do not yet assume that $f$ has non-vanishing Gaussian curvature. 
\begin{lma}\label{lma: startingdecomp}
    Let $\mu$ be a self-similar measure on $\RR^k$, let $f \colon \RR^k \to \RR^d$ be $C^2$, and let $T(\bx) = (\bx,f(\bx))$. 
    Then for sufficiently large $n \in \NN$ and $\bxi \in \RR^{k+d}$ with first $k$ coordinates being $0$, 
    \[ |\widehat{\mu_T}(\bxi)| \leq \sum_{D_n}\sum_{\mu_{\omega,n}\sim D_n} \left|\widehat{\mu}_{\omega,n}(\bxi_{\omega,n})\right|+O_{\mu,f}(|\bxi|/2^{2n}). \]
\end{lma}
\begin{proof}
    Consider the dyadic decomposition $\Lambda_{D_n}$ for $D_n\in\cD_n$ for $n\geq 1000$. 
    We obtain 
\begin{align}\label{e:initialbreakdown}
\begin{split}
\widehat{\mu_T}(\bxi)&=\int e^{-2\pi i \langle \by,\bxi \rangle}d\mu_T(\by)\\
&=\sum_{D_n\in\cD_n} \sum_{\mu_{\omega,n}: \mu_{\omega,n}\sim D_n}\int e^{-2\pi i \langle \by,\bxi \rangle}d\mu_{\omega,n,T}(\by),
\end{split}
\end{align}
where $\mu_{\omega,n,T}=(\mu_{\omega,n})_T$. 

Our aim is to estimate for each $\mu_{\omega,n}$ the integral
\[
\int e^{-2\pi i (\by,\bxi)}d\mu_{\omega,n,T}(\by).
\]
Consider the measure $\mu'_{\omega,n,T}$ obtained by pushing $\mu_{\omega,n,T}$ to $T_{\omega,n}$ by the affine projection keeping the first $k$ coordinates fixed. 
More precisely, $\mu'_{\omega,n,T}$ is the image of $\mu_{\omega,n,T}$ under the map sending $(x_1,\dotsc,x_k,x_{k+1},\dotsc,x_{k+d}) \in \RR^{k+d}$ to the unique element of $\{(x_1,\dotsc,x_k)\} \times \RR^d$ which lies on $T_{\omega,n}$. 
Observe that as long as $|\bxi|\leq 2^{2n}$,
\begin{align}
\begin{split}\label{e:linearise}
&\left|\int e^{-2\pi i \langle \by,\bxi \rangle }d\mu_{\omega,n,T}(\by)-\int e^{-2\pi i \langle \by,\bxi \rangle}d\mu'_{\omega,n,T}(\by)\right|\\
&=\left|\int e^{-2\pi i \langle (\bx,f(\bx)),\bxi \rangle }d\mu_{\omega,n}(\bx)-\int e^{-2\pi i \langle (\bx,f(\bx_0)+\nabla f(\bx_0)\cdot(\bx-\bx_0)),\bxi \rangle}d\mu_{\omega,n}(\bx)\right|\\
&=\left| \int e^{-2\pi i \langle (\bx,f(\bx_0)+\nabla f(\bx_0)\cdot(\bx-\bx_0)),\bxi \rangle}(e^{-2\pi i \langle (\bx,f(\bx)-f(\bx_0)-\nabla f(\bx_0)\cdot(\bx-\bx_0)),\bxi\rangle }-1)d\mu_{\omega,n}(\bx)\right|\\
&=O\left(\int |e^{-2\pi i \langle (\bx,f(\bx)-f(\bx_0)-\nabla f(\bx_0)\cdot(\bx-\bx_0)),\bxi \rangle}-1|d\mu_{\omega,n}(\bx)\right)\\
&=O(|\mu_{\omega,n}||\bxi|/2^{2n}),
\end{split}
\end{align}
using Taylor's theorem and recalling that $|\mu_{\omega,n}|$ is the total mass of $\mu_{\omega,n}$. 
We write $\bxi=\bxi_\perp+\bxi_\parallel$ where $\bxi_\perp$ is perpendicular to $T_{\omega,n}$ and $\bxi_\parallel$ is parallel to $T_{\omega,n}$. 
Then
\begin{align}\label{e:rescaleft}
\begin{split}
\Big| \int e^{-2\pi i \langle \by,\bxi \rangle}d\mu'_{\omega,n,T}(\by) \Big| &= \Big| \int e^{-2\pi i \langle \by-\by_0,\bxi \rangle }d\mu'_{\omega,n,T}(\by-\by_0) \Big| \\
&=\Big| \int e^{-2\pi i \langle \by-\by_0,\bxi_\parallel \rangle}d\mu'_{\omega,n,T}(\by-\by_0) \Big|\\
&= \Big| \widehat{\mu}_{\omega,n}(\bxi_{\omega,n}) \Big|.
\end{split}
\end{align}
Using~\eqref{e:linearise} and~\eqref{e:rescaleft}, the result follows from~\eqref{e:initialbreakdown}. 
\end{proof}

\subsection{Separation of tangents}
Let $f\colon \RR^k \to\mathbb{R}$ be a $C^2$ function with non-vanishing Hessian determinant over a compact set $K$. 
Then the manifold $\Gamma_f \coloneqq \{(\bx,f(\bx)) : \bx\in\mathbb{R}^k\}$ has non-vanishing Gaussian curvature over $K$. Let $n\geq 1000$ and consider the decomposition $\cD_n$. For each $D_n\in\cD_n$, let $x_{D_n}$ be its centre. Consider the tangent plane $T_{D_n}=T_{x_{D_n}}$ of $\Gamma_f$ on the point $x_{D_n}$. We write $T^{\perp}_{D_n}$ to indicate the normal direction to $T_{D_n}$ in real projective space $\bR\bP^k$ which we equip with the natural metric induced from the sphere $\mathbb{S}^{k}  \subset \RR^{k+1}$. 
\begin{lma}\label{lma: bounded overlaps}
    Let $n\geq 1000$. Consider all the dyadic cubes such that $D_n\cap K\neq\varnothing$. Consider the corresponding discrete set $\{T^\perp_{D_n}\}_{D_n}\subset\bR\bP^k$ for $D_n$ ranging over this collection $\mathcal{C}$. Then the union of balls
    \[
    \bigcup_{D_n \in \mathcal{C}} B_{2^{-n}}(T^\perp_{D_n})
    \]
    has bounded multiplicity with a bound that depends on $f$.
\end{lma}
\begin{proof}
Since $\Gamma_f$ has non-vanishing Gaussian curvature over the compact set $K$, we see that for large enough $n$, the Jacobian determinant of the shape operator, and therefore also the principal curvatures (which are continuous functions), are uniformly bounded from below for all $D \in \cD_n$ such that $D\cap K\neq\varnothing$. 
Therefore for two adjacent cubes $D,D'\in \cD_n$, the distance between $T^\perp_{D}$ and $T^\perp_{D'}$ on $\bR \bP^k$ is $\gg 2^{-n}$. 
The implicit constant here depends on the second-order derivatives of $f$. Following this argument, we see that there is a constant $c>0$ such that for each $\bx\in K$, for all $\by\in K$ with $|\by-\bx|<c$,
\[
d(T^\perp_{\bx},T^\perp_{\by})_{\bR \bP^k}\gg |\bx-\by|.
\]
This proves the result with $K$ being replaced by $K\cap B_c$ where $B_c$ is any $c$-ball. Since $K$ is compact, we can cover $K$ with finitely many such $c$-balls. From here the lemma is proved. 
\end{proof}

\subsection{Proof in the homogeneous case}

We in fact prove the following more general quantitative estimate for Fourier decay involving the Fourier $l^p$ dimension (recall~\eqref{e:lpdef}). 
Unfortunately there has not yet been much study of the Fourier $l^p$ dimension of self-similar measures when $1 < p < 2$ strictly. 
Theorem~\ref{thm: image fourier decay} follows immediately from the $p \in \{1,2\}$ cases of the following theorem (together with Lemma~\ref{lma: correlationdim}). 
\begin{thm}\label{thm: allparams}
Under the assumptions of Theorem~\ref{thm: image fourier decay}, if $1 \leq p \leq 2$ satisfies $d_{p/(p-1)} + \kappa_p > k$ then $|\widehat{\mu_f}(\xi)|\ll |\xi|^{-\sigma}$ with arbitrary $0<\sigma < \sigma_p$, where 
\[ \sigma_p \coloneqq \frac{d_{p/(p-1)} + \kappa_p - k}{2p - k + 2\kappa_* + \kappa_p - d_{p/(p-1)}}. \] 
When $p=1$, $d_{p/(p-1)}$ means the Frostman exponent $d_{\infty}$. 
\end{thm} 

\begin{lma}\label{lma: homogquant}
Theorem~\ref{thm: allparams} holds under the additional assumption that $\mu$ is homogeneous. 
\end{lma}
\begin{proof}
Recall that $T\colon \bx \mapsto (\bx,f(\bx))$, and recall the definition of $\bxi_{\omega,n}$ from~\eqref{e:definexiomegan}. 
Fix $1 \leq p \leq 2$ and $1 < \gamma < 2$, and let $q$ be such that $1/p + 1/q = 1$ (with $q=\infty$ if $p=1$). 
Let $\bxi=(0,\dots,0,\xi_{k+1})$ be such that $|\xi_{k+1}| = 2^{\gamma n}$ for some positive integer $n$. 
We apply Lemmas~\ref{lma: lifttograph} and~\ref{lma: startingdecomp} in the $d=1$ case, giving 
\[ |\widehat{\mu_f}(\xi_{k+1})| = |\widehat{\mu_T}(\bxi)| \leq \sum_{D_n \in \mathcal{D}_n}\sum_{\mu_{\omega,n}\sim D_n} \left|\widehat{\mu}_{\omega,n}(\bxi_{\omega,n})\right|+O(|\bxi|/2^{2n}). \]
As $\mu$ is homogeneous, for a given $n$, $\mu_{\omega,n}$ are all at exactly the same scale and orientation (related to the scale and orientation of $\mu$). 
Using the self-similarity of $\mu$, 
\begin{equation}\label{eqn: measure decomposition}
|\widehat{\mu_T}(\bxi)| \leq \sum_{D_n}\sum_{\mu_{\omega,n}\sim D_n} |\mu_{\omega,n}| |\widehat{\mu}(r_n O_n(\bxi_{\omega,n}))| + O(|\bxi|/2^{2n}),
\end{equation}
where $r_n\asymp 2^{-n}$ is the contraction ratio of the copy $\mu_{\omega,n}$ and $O_n^{-1} \in O_k(\mathbb{R})$ is its orientation (note that these are all the same for each $\mu_{\omega,n}$). 
We write $\bxi_{D_n} \coloneqq O_n(\bxi_{\omega,n})$ since there is no dependence on $\omega$ as long as $\mu_{\omega,n}\sim D_n$. 
By Lemma~\ref{lma: compactly supported}, for all $\varepsilon > 0$ we can bound 
\begin{equation}\label{e:usingcompactlysupported}
|\widehat{\mu}(r_n \bxi_{D_n})| \ll \int_{B_{2^{\varepsilon n}}(r_n \bxi_{D_n})} |\widehat{\mu}(\bxi')| d\bxi' + 2^{-100k n}. 
\end{equation}
Uniformly for all $D_n,\bxi$ under consideration, $T^\perp_{D_n}$ is bounded away from being orthogonal to $\bxi$. Thus $|\bxi_{D_n}|\ll |\bxi|$ where the implicit constant does not depend on $D_n$ and $\bxi$. Therefore each $\bxi_{D_n} r_n$ is $\ll |\bxi|/2^n$ away from the origin. 

Since $(\bx,f(\bx))_{\bx\in\mathbb{R}^k}$ has non-vanishing Gaussian curvature over $\supp(\mu)$, we see that for different $D_n$, the $\bxi_{D_n}r_n$ are $|\bxi| 2^{-2n}$ separated with bounded multiplicity. 
For this, we mean that the discrete set
\begin{equation*}
F \coloneqq \{\bxi_{D_n}r_n\}_{D_n\in\cD_n}\subset\mathbb{R}^k
\end{equation*}
can be divided into disjoint cubes of size $|\bxi|/2^{2n}$ with the property that each cube contains at most a bounded number (independent of $n$) of points in $F$ (see Lemma~\ref{lma: bounded overlaps}). 
Since $\gamma < 2$ (so $|\bxi|\ll 2^{2n}$) the points $\bxi_{D_n}r_n$ are not necessarily $\gg 1$ separated. 
Thus it is a non-trivial problem to estimate the concentration of those points in unit-scale cubes.

\begin{claim}\label{claim: assouad}
Let $D$ be any cube of sidelength $1$. For each $\kappa_*' > \kappa_*$, uniformly for all $\bxi$ such that $2^n\leq |\bxi|\leq 2^{2n}$ we have 
\[
\# ( \{\bxi_{D_n}r_n\}_{D_n\in\cD_n}\cap D ) \ll (2^{2n}/|\bxi|)^{\kappa_*'}.
\]
\end{claim}

\begin{proof}[Proof of claim]
First note that there is some cube $C$ (which need not be a dyadic cube) of size $\asymp 2^n/|\bxi|$ such that 
\[
\# (\{\bxi_{D_n}r_n\}_{D_n\in\cD_n}\cap D) \ll \#( \{\bxi_{D_n}/|\bxi|\}_{D_n\in\cD_n} \cap C ).
\]
The set $\{\bxi_{D_n}/|\bxi|\}_{D_n\in\cD_n}$ is $\{\Beta_{D_n}\}_{D_n\in\cD_n}$ for the unit vector $\Beta=(0,\dots,0,1)$. For each $\bx\in\mathbb{R}^k$, let $T_\bx$ be the tangent space of $\Gamma_f$ at $(\bx,f(\bx))$. Let $P(\bx)$ be the projection of $\Beta$ to $\mathbb{R}^k\times\{0\}$ in the direction that is orthogonal to $T_\bx$. 
This map $\bx \mapsto P(\bx)$ is $C^1$, and near each $\bx$ in a neighbourhood of $\supp(\mu)$ it is a local $C^1$ diffeomorphism by the inverse function theorem. 
By Lemma~\ref{lma: bounded overlaps}, points $\Beta_{D_n}$ are $2^{-n}$ separated with bounded multiplicity. 
This implies that $\# (\{\Beta_{D_n}\}_{D_n\in\cD_n} \cap C)$ is at most some bounded number times the maximal cardinality of a $2^{-n}$-separated set $A$ in $\supp(\mu)$ with the property that $P(A)\subset C$. Using Lemma~\ref{lma: bounded overlaps}, this maximal cardinality is then
\[
\ll \max_{\bx\in\mathbb{R}^k}N_{2^{-n}}( B_{2^n/|\bxi|}(\bx)\cap \supp(\mu))\ll (2^{2n}/|\bxi|)^{\kappa_*'}
\]
by the definition of the Assouad dimension $\kappa_*$. This proves the claim.
\end{proof}

We return to the proof of Lemma~\ref{lma: homogquant}. 
Fix $\kappa'_*>\kappa_*$, $d_q' < d_{q}$ and $\kappa'_p<\kappa_p$ very close to the respective exponents. 
We assume $p>1$; the case $p=1$ can be handled separately. 
The following important calculation starts from~\eqref{eqn: measure decomposition}, using~\eqref{e:usingcompactlysupported} and the fact that $\# F \ll 2^{kn}$: 
\begin{align*}
    &|\widehat{\mu_T}(\bxi)| \ll \sum_{D_n}\sum_{\mu_{\omega,n}\sim D_n} |\mu_{\omega,n}| \int_{B_{2^{\varepsilon n}}(r_n \bxi_{D_n})} |\widehat{\mu}(\bxi')| d\bxi' + 2^{kn} 2^{-100 k n} + 2^{\gamma n} 2^{-2n} \\ 
    &\ll \left( \sum_{D_n} \left( \sum_{\mu_{\omega,n}} |\mu_{\omega,n}|\right)^q \right)^{1/q} \left( \sum_{D_n} \left( \int_{B_{2^{\varepsilon n}}(r_n\bxi_{D_n})} |\widehat{\mu}(\bxi')| d\bxi' \right)^p \right)^{1/p}  + 2^{-(2-\gamma)n}; 
\end{align*}
the last line was by H\"older's inequality. 
Now we can apply Jensen's inequality twice. We can apply it to the convex function $t \mapsto t^q$ and use that each $D_n$ can intersect supports of measures $\mu_{\omega,n}$ from at most the $3^k$ surrounding cubes, so we can bring in the $L^q$ dimensions. 
We can also apply Jensen to $t \mapsto t^p$: 
\begin{align*}
    |\widehat{\mu_T}(\bxi)| &\ll \left( 3^{k(q + 1)} \sum_{D_n} (\mu(D_n))^q \right)^{1/q}\\* 
    &\times \left(  \sum_{D_n} (\vol(B_{2^{\varepsilon n}}(r_n\bxi_{D_n})))^{p-1} \int_{B_{2^{\varepsilon n}}(r_n\bxi_{D_n})} |\widehat{\mu}(\bxi')|^p d\bxi' \right)^{1/p} + 2^{-(2-\gamma)n}. 
\end{align*}
Each fixed frequency $\bxi'$ is within a distance of $2^{\varepsilon n}$ from $\ll 2^{\varepsilon k n}$ cubes of size $1$, each of which contains $\ll (2^{2n}/|\bxi|)^{\kappa_*'}$ points in $F$ by Claim~\ref{claim: assouad}, so $\bxi'$ appears in $\ll (2^{2n}/|\bxi|)^{\kappa_*'} 2^{\varepsilon k n}$ of the integrals in the previous line. Therefore, possibly increasing $\varepsilon$ if necessary, for some constant $C>0$, 
\begin{align}
    |\widehat{\mu_T}(\bxi)| &\ll 2^{-d_q' \cdot n\cdot (q-1)/q}  \left( \left(\frac{2^{2n}}{|\bxi|}\right)^{\kappa_*'} 2^{\varepsilon k n} \int_{\mbox{dist}(\bxi',F) \leq 2^{\varepsilon n}} |\widehat{\mu}(\bxi')|^p d\bxi' \right)^{1/p}  + 2^{-(2-\gamma)n} \label{i:useclaim} \\ 
    &\ll 2^{-d_q' n/p}   \left( \left(\frac{2^{2n}}{|\bxi|}\right)^{\kappa_*'} 2^{\varepsilon k n} \int_{B_{C|\bxi|/2^n}(\mathbf{0})} |\widehat{\mu}(\bxi')|^p d\bxi' \right)^{1/p}  + 2^{-(2-\gamma)n} \nonumber \\
    &\ll 2^{(-d_q' + (2-\gamma)\kappa_*' + \varepsilon k + (k-\kappa_p')(\gamma - 1))n/p} + 2^{-(2- \gamma)n} \label{i:uselp}, 
\end{align} 
where in the last two lines we used that points in $F$ are $\ll |\bxi|/2^n$ away from the origin, and then used the definition of $l^p$ dimension. 

Now, the assumption $\kappa_2 > k/2$ and the fact that $\kappa_2 = d_2$ tell us that $\sigma_2 > 0$. 
Assume we had chosen close enough approximations so that $d_q' + \kappa_p' > k$. 
Then as $\gamma$ increases, the first term in~\eqref{i:uselp} will decrease, while the second term will clearly increase. 
    After some algebraic manipulation, the value of $\gamma$ at which both expressions are approximately equal is 
    \[ \gamma \coloneqq \frac{2p - k -d_{q}' + 2\kappa_*' + \kappa'_p}{p-k+\kappa'_p+\kappa'_*}, \]
    hence
    \[
    \frac{2-\gamma}{\gamma} = \frac{d_{q}' + \kappa_p' - k}{2\kappa_*' + \kappa_p' - d_q' + 2p - k}.
    \]
Since $d_q' < d_q \leq d_2 = \kappa_2 \leq \kappa_* < \kappa_*'$, we have $\gamma > 1$, and since $d_q' + \kappa_p' > k$ we indeed have $\gamma < 2$. 
From the definition of $\sigma_p$ we see that $(2 - \gamma)/\gamma$ can be made arbitrarily close to $\sigma_p$. 
    Therefore 
    \[
    |\widehat{\mu_T}(\bxi)| \ll 2^{-(2 - \gamma -\varepsilon)n} = |\bxi|^{-(2 - \gamma - \varepsilon)/\gamma} \ll |\bxi|^{-(\sigma_p - \varepsilon)}.
    \]

    Finally, increasing the implicit constant if needed, we have shown that $|\widehat{\mu_f}(\xi_{k+1})| \ll |\bxi|^{- (\sigma_p - \varepsilon)}$ for $|\bxi| \in [2^{\gamma n - 9},2^{\gamma n + 9}]$, and so $|\widehat{\mu_f}(\xi_{k+1})| \ll |\bxi|^{- (\sigma_p - \varepsilon)}$ as $|\bxi| \to \infty$. 
    Since $p \in [1,2]$ and $\varepsilon > 0$ were arbitrary, this completes the proof. 
\end{proof}

\subsection{Proof in the non-expanding case}\label{ss: nonhomog}
We now describe how the proof of Lemma~\ref{lma: homogquant} can be adapted to the case where we only assume that $\mu$ is non-expanding. 
\begin{proof}[Proof of Theorem~\ref{thm: allparams}]
We cannot arrive at~\eqref{eqn: measure decomposition} because there is no uniform $r_n$ in this case. We step back to Lemma~\ref{lma: startingdecomp}. 
In this case $\mu_{\omega,n}$ are not at the same scale and orientation. Let $L_n$ denote the set of all possible scales and orientations of $\mu_{\omega,n}$. Then we see that for each $\epsilon>0$,
\begin{equation*}
\#L_n\ll 2^{\epsilon n}.
\end{equation*}
For each $g\in L_n$, we can apply the argument for the collection of $\mu_{\omega,n}$ that has the scale and orientation determined by $g$. 
Abusing notation slightly, we also write this collection as $g$. We obtain the following upper bound: 
\begin{equation}\label{e:expandingmaincalc}
|\widehat{\mu_T}(\bxi)| \ll \sum_{g\in L_n}\sum_{D_n}\sum_{\substack{\mu_{\omega,n}\in g \\ \mu_{\omega,n}\sim D_n}} |\mu_{\omega,n}| |\widehat{\mu}(r_{n,g} O_{n,g}(\bxi_{\omega,n}))| + |\bxi|/2^{2n}. 
\end{equation}
As in~\eqref{e:usingcompactlysupported}, we can use Lemma~\ref{lma: compactly supported} to see that 
\begin{equation}\label{e:expandinglargerball}
|\widehat{\mu}(r_{n,g} O_{n,g}(\bxi_{\omega,n}))| \ll \int_{B_{2^{\varepsilon n}}(r_{n,g} O_{n,g}(\bxi_{\omega,n}))} |\widehat{\mu}(\bxi')| d\bxi' + 2^{-100k n}. 
\end{equation}
Combining~\eqref{e:expandingmaincalc} and~\eqref{e:expandinglargerball}, the rest of the arguments are the same as in the case when $\mu$ is homogeneous. 
Since $\varepsilon$ can be taken arbitrarily small, this finishes the proof.
\end{proof}

\begin{rem}\label{rem: loss}
Several estimates were made in the above proofs which are not sharp, even if the self-similar measure is assumed to be AD-regular and satisfy the SSC, and this means we cannot prove that $\sigma$ can be chosen to be close to~$\kappa_2/2$. 
    In particular, there can be loss when~\eqref{e:linearise} approximates $\mu_T$ locally by a suitable linear image of $\mu$ and when Lemma~\ref{lma: startingdecomp} accordingly bounds $\widehat{\mu_T}$ as a sum of magnitudes of Fourier transforms of the linear images. 
    In specific cases, it might be possible to drop the linearisation part of the error. 
    However, another place where there is potential loss is when we apply H\"older's inequality and bound the resulting integral of $\widehat{\mu}$ over the neighbourhood of a fractal set using the $l^p$ dimension, see Section~\ref{ss:fup}. 
\end{rem}

\section{Nonlinear arithmetic}\label{sec: arithmetic}
\subsection{Background and results}
The results in this section are motivated by the following problem. Given a fractal set $K$, how can one tell if $K\cdot K$ has positive Lebesgue measure? Here we used the following notion for the \emph{arithmetic product} of subsets $K_i \subset \R$: 
\begin{equation*}
K_1\cdot K_2 \dotsb K_k \coloneqq \Big\{\prod_{i=1}^k x_i : x_1 \in K_1,\dotsc,x_k\in K_k\Big\}.
\end{equation*}
A more general problem is to consider the image $f(K\times K)$ for a certain smooth map $f\colon \mathbb{R}^2\to \mathbb{R}$. Such a problem was studied in~\cite{EGRS75} from number theory, in~\cite{MY,N79} from smooth dynamical systems and in~\cite{Fu2,Mar} from geometric measure theory. Recently, many new results have been proved. 
More precisely, we have~\cite{BY19,Y20a, Y20} extending~\cite{EGRS75}; \cite{SimonTaylor, T, T1, T2, Ya,ZRZJ} extending~\cite{MY,N79}; \cite{HS12, H14, Sh, Wu} extending \cite{Fu2}; and \cite{KO, O19, YuRadial, YuManifold} extending~\cite{Mar}. 
This is certainly not a complete list. For example, we have omitted great achievements regarding Falconer's distance conjecture; see~\cite{SW} for a recent account.

Our starting point is the following general result due to Marstrand. We do not state the most general version (which requires the notion of the transversal family as in~\cite{PS00}). In fact, we will state the theorem as a high-level idea and then supply two concrete examples.

\begin{thm}[Peres--Schlag~\cite{PS00}, (generalised) Marstrand projection theorem, informal statement]
Let $\{P_a\}_a$ be a parametrised family of smooth maps $\mathbb{R}^2\to\mathbb{R}$ (satisfying certain conditions). 
Let $A\subset\mathbb{R}^2$ be a compact set with $\Haus A>1$. 
Then for almost all parameters $a$ (with respect to some natural measure), $P_a(A)$ has positive Lebesgue measure. 
\end{thm}
\begin{exm}
    For $a\in [0, 2\pi)$, we can define $P_a\colon (x,y)\mapsto x\cos a+y\sin a$. 
    It is the orthogonal projection in the direction given by $a$. The corresponding result is then the classical Marstrand's projection theorem for linear projections.
\end{exm}
\begin{exm}
    For $\ba\in\mathbb{R}^2$, we can define $P_{\ba}\colon \bx\mapsto (\bx-\ba)/|(\bx-\ba)|\in\mathbb{S}^1$. This is the radial projection centred at $\ba$. The corresponding result is then the classical Marstrand projection theorem for radial projections.
\end{exm}

For general sets, there is no hope of upgrading Marstrand's projection theorems to cover \emph{all} parameters. However, if one assumes that $K$ has additional structure then the result can often be improved, as we see in \cite[Corollary~2.9 and Theorem~1.3]{BaranyArithmetic} and~\cite{YuRadial, YuManifold}. 
One interesting structure for us to explore is coded as the notion of self-similarity. 
In this direction, one result that is known is the following. 
Recall that we made the standing assumption that the contractions in an IFS never share a common fixed point. 

\begin{thm}[follows from Hochman--Shmerkin~\cite{HS12}\footnote{The details of how this follows from~\cite{HS12} are described in {\color{cyan}\textit{https://mathoverflow.net/questions/132445/arithmetic-products-of-cantor-sets}}, accessed 4 May 2025. The rational incommensurability assumption is likely to be redundant due to \cite[Theorem~1.23]{H10} (this is certainly true in the $k=2$ sets case) but we do not pursue this topic in this article.}]\label{thm: hochmanshmerkin}
    Let $E_1,\dotsc,E_k$ be self-similar sets in $\RR$. Suppose that for each $E_i$ there is some contraction ratio $r_i$ such that $\{ \log r_1,\dotsc,\log r_k\}$ are linearly independent over $\mathbb{Q}$ (in other words $E_1,\dotsc,E_k$ are rationally incommensurable). 
    Then 
    \[ \Haus (E_1 \cdot E_2 \dotsb E_k)=\min\{\Haus E_1+ \dotsb + \Haus E_k,1\}. \]
\end{thm}

Unlike Marstrand's projection theorem, there is no uncertainty in Theorem~\ref{thm: hochmanshmerkin}. The idea behind its proof (in the $k=2$ case for simplicity) relies on relating the arithmetic product set $E\cdot F$ to $\log E+\log F$ and then to the sum set $E+F$, which is a certain linear projection of $E\times F$. 
Marstrand's projection theorem cannot say much about this specific projection. 
However, under the assumption, $E+F$ will have the required dimension because at small scales, $E$ and $F$ look quite different and there should be no additive collision. 
As far as $E\cdot F$ is concerned, $\log E,\log F$ should have no common `additive structure' even if they all have the same contraction ratios! This leads us to believe that even stronger results can hold. 

\begin{conj}\label{conj: multiplication}
     Let $E_1,\dotsc,E_k$ be $k\geq 2$ self-similar sets in $\RR$. 
     If \[ \sum_{i=1}^k\Haus E_i>1,\] then $E_1\cdot E_2 \dotsb E_k$ has positive Lebesgue measure. 
\end{conj}
A main goal of this section is to use Fourier decay to work towards Conjecture~\ref{conj: multiplication} and find dimension conditions that guarantee that arithmetic products have not only full dimension but also positive measure. 

One can formulate an analogous conjecture for \emph{multiplicative convolutions} of measures. Given Borel probability measures $\mu_i$ on $\RR$, their multiplicative convolution is defined by the pushforward 
\begin{equation*}
\mu_1\cdot \mu_2 \dotsb \mu_k \coloneqq (\mu_1\times\dots\times \mu_k)_m,
\end{equation*}
where $m \colon \mathbb{R}^k\to \mathbb{R}$ is the multiplication $(x_1,\dots,x_k) \mapsto \prod_{i=1}^k x_i$. 
\begin{conj}\label{conj:multconv}
    Let $\mu_1,\dots,\mu_k$ be self-similar measures supported in $(0,\infty)$ with $\sum_{i=1}^k\kappa_2(\mu_i)>1$. 
    Then $\mu_1\cdot \mu_2 \dotsb \mu_k$ is absolutely continuous with respect to Lebesgue measure. 
\end{conj} 

These conjectures are wide open. 
For certain specific self-similar sets $F$, such as the middle-third Cantor set, it has been proved (without using Fourier analysis) that $F \cdot F$ has positive Lebesgue measure and moreover contains non-trivial intervals~\cite{AthreyaCantor}. Related results have also been obtained under the assumption that the sets have large enough `thickness,' which quantifies size in a different way to Hausdorff dimension, see \cite[Proposition~2.9]{SimonTaylor}. 
Moreover, a recent result in~\cite{OSS} says that for Borel probability measures $\mu_1,\dots,\mu_k$ with $\kappa_2(\mu_1)+\dots+\kappa_2(\mu_k)>1$ the multiplicative convolution measure $\mu_1\cdot \mu_2\cdot \dots\cdot \mu_k$ has polynomial Fourier decay. The point there is that $\mu_1,\dots,\mu_k$ are not required to be self-similar. 

We will work towards Conjectures~\ref{conj: multiplication} and~\ref{conj:multconv}, deferring the proofs to Section~\ref{ss:nonlinearproofs} to avoid disrupting the narrative flow. 
Before stating our results, we state the following lemma which records well-known facts demonstrating why Fourier decay is useful; this lemma will be used repeatedly in our proofs. 
\begin{lma}[Theorem~3.12 in \cite{WolffLectures} and Theorem~3.4 in \cite{Ma2}]\label{lma: abscts}
    Let $\mu$ be a Borel probability measure on $\RR^k$. Then 
    \begin{itemize}
        \item If $\int_{\RR^k} |\widehat{\mu}(\bxi)|^2 d\bxi < \infty$ then $\mu$ is absolutely continuous with respect to Lebesgue measure and its density (i.e. its Radon--Nikodym derivative) is an $L^2$-function. 
        \item If $\int_{\RR^k} |\widehat{\mu}(\bxi)| d\bxi < \infty$ then $\mu$ is absolutely continuous with a density that is a continuous function. 
    \end{itemize}
\end{lma}

We are ready to state our main result on nonlinear arithmetic of self-similar sets, which immediately implies Theorem~\ref{thm: arith headline}. 
\begin{thm}\label{thm: enough variables mul}
Let $E,F,G$ be self-similar sets in $\RR$. 
\begin{itemize}
    \item If 
    \begin{equation}\label{e:twoselfsim}
    \Haus E \cdot \Haus F+\max\{1.5\Haus E+\Haus F,1.5\Haus F+\Haus E\}>2.5 ,
    \end{equation}
    then $E \cdot F$ has positive Lebesgue measure. 
    \item
    If 
    \[ \frac{1}{2}\Haus E + \frac{1}{2} \Haus F + \frac{\Haus G - 0.5}{\Haus G + 1.5}  > 1  \] 
    then $E\cdot F\cdot G$ has non-empty interior. 
    \end{itemize}
\end{thm}

The bound~\eqref{e:twoselfsim} from Theorem~\ref{thm: enough variables mul} is illustrated in Figure~\ref{fig: twoselfsim}. 
\begin{figure}[ht]
	\centering
	\includegraphics[width=0.5\textwidth]{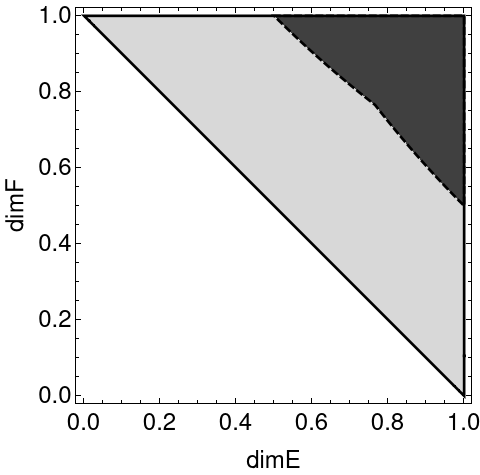}
	\caption{The dark shaded region describes the sufficient condition from~\eqref{e:twoselfsim} for the arithmetic product of two self-similar sets to have positive Lebesgue measure. The light shaded region is conjectured to be sufficient in Conjecture~\ref{conj: multiplication}.}
	\label{fig: twoselfsim}
\end{figure}

We emphasise that in Theorem~\ref{thm: enough variables mul}, we do not assume regularity (or separation) conditions.\footnote {We hide a technical point here that it is extremely difficult to determine the Hausdorff dimension of a self-similar set without any separation conditions. However, under the ESC, which holds extremely generically, Hausdorff dimension coincides with the minimum of 1 and the similarity dimension~\cite{H14}, a quantity which is much easier to compute.}
\begin{rem}
    Similar results to Theorems~\ref{thm: enough variables mul} and~\ref{thm: arith headline} were obtained by Yu in~\cite{YuRadial,YuManifold} under conditions on the Fourier $l^1$-dimensions of measures. Although some of the thresholds there are better than the thresholds here, we note that Hausdorff dimension is in general larger in value and easier to compute than $l^1$-dimension.
\end{rem}
Of course, it is also possible to explicitly work out sufficient conditions involving more than three sets. We do not continue this route. 

We will in fact prove Theorem~\ref{thm: enough variables mul} using the following analogous result for AD-regular self-similar measures. 
\begin{thm}\label{thm: measure main}
    Let $\mu$, $\nu$, $\gamma$ be AD-regular self-similar measures supported in $(0,\infty)$. 
    \begin{itemize}
    \item If 
    \begin{equation}\label{e:twomeascondition}
    \kappa_2( \mu)\kappa_2(\nu)+\max\{1.5\kappa_2(\mu)+\kappa_2(\nu),1.5\kappa_2(\nu)+\kappa_2(\mu)\}>2.5
    \end{equation}
    then the multiplicative convolution $\mu \cdot \nu$ is absolutely continuous with an $L^2$-density. 
    \item If 
    \begin{equation}\label{e:threemeascondition}
    \frac{\kappa_2(\mu)}{2} + \frac{\kappa_2(\nu)}{2} + \frac{\kappa_2(\gamma) - 0.5}{\kappa_2(\gamma) + 1.5} >1, 
\end{equation}
then $\mu \cdot \nu \cdot \gamma$ is absolutely continuous with a continuous density. 
\end{itemize}
\end{thm}
The obvious analogue of Theorem~\ref{thm: arith headline} for AD-regular self-similar measures holds, but we will not state it. 

    One can also deduce results for measures without the AD-regularity assumption although the dimensional thresholds will be a bit worse. For instance, the following result for two measures holds; the $7/9 = 0.777\dots$ can be compared with the $0.765\dots$ from Theorem~\ref{thm: arith headline}. 
\begin{prop}\label{prop: escarith}
    Let $\mu$ and $\nu$ be self-similar measures supported in $(0,\infty)$. 
    Suppose that at least one of the following holds: 
    \begin{itemize}
    \item $\min\{\kappa_2(\mu),\kappa_2(\nu)\} > 7/9 = 0.777\dots$,
    \item $\max\{4\kappa_2(\mu) + 5\kappa_2(\nu), 5\kappa_2(\mu) + 4\kappa_2(\nu)\} > 7$,
    \item $\nu$ is AD-regular and $2\kappa_2(\mu) \kappa_2(\nu) + 3\kappa_2(\mu) + 2\kappa_2(\nu) > 5$.
    \end{itemize}
    Then $\mu \cdot \nu$ has an $L^2$-density. 
\end{prop}

The function $f\colon (x_1,\dotsc,x_k)\mapsto x_1\dotsb x_k$ is by no means the only function to which these methods can be applied. 
For an example of a different function, for fixed $(a,b)\in\mathbb{R}^2$ consider the radial projection $R_{(a,b)} \colon \RR^2 \setminus \{(a,b)\} \to \mathbb{S}^1$ given by 
\[ R_{(a,b)}\colon (x,y)\mapsto \frac{(x-a,y-b)}{|(x-a,y-b)|}. \] 
If $x>a$ and $y>b$, after some smooth identification using the arctan function, we can consider the map
\[
L_{(a,b)}(x,y)=\log (x-a)-\log (y-b).
\]
This map has Hessian matrix
\[
\begin{bmatrix}
    \frac{-1}{(x-a)^2} &  0\\
    0 & \frac{1}{(y-b)^2}
\end{bmatrix},
\]
so since $x>a$ and $y>b$, the graph of $L_{(a,b)}$ has non-vanishing Gaussian curvature. 
The following result follows by using the same arguments as in Theorems~\ref{thm: measure main} and~\ref{thm: enough variables mul}.

\begin{thm}\label{thm: radial} 
      Let $E,F$ be self-similar sets in $\RR$, and fix $(a,b)\in \RR^2$. 
      If 
      \[\Haus E \cdot \Haus F+\max\{1.5\Haus E+\Haus F,\Haus E+ 1.5\Haus F\}>2.5\] 
      (which holds if $\min\{\Haus E,\Haus F \} > (\sqrt{65}-5)/4$), then $R_{(a,b)}((E\times F) \setminus \{(a,b)\})$ has positive Lebesgue measure. 
\end{thm}
\begin{rem}
    The generalised Marstrand projection theorem tells us that the conclusion of Theorem~\ref{thm: radial} holds for Lebesgue almost all $(a,b)\in\mathbb{R}^2$. 
    Here, we took care of the exceptions. 
\end{rem}

Finally, the methods used in this section together with Theorem~\ref{thm: image fourier decay} can also be used to prove results about images of self-similar measures on $\RR^k$ for sufficiently large $k$. For example, the following statement holds. 
\begin{thm}\label{thm: higher to one}
    Fix any integer $k \geq 5$, and let $f\colon \mathbb{R}^k\to\mathbb{R}$ be smooth with non-zero Hessian determinant, e.g. $f(x_1,\dotsc,x_k) = x_1^2 + \dotsb + x_k^2$. 
    Let $\mu$ be an $s$-AD-regular non-expanding self-similar measure on $\mathbb{R}^k$ with 
    \[
    s > 2+\frac{k}{2}.
    \]
 Then $\mu_f$ has an $L^2$-density. 
 
 In particular, the image under $f$ of any non-expanding self-similar set $E\subset\mathbb{R}^k$ with $\Haus E>2+(k/2)$ has positive Lebesgue measure. 
\end{thm}

\subsection{Nonlinear arithmetic proofs}\label{ss:nonlinearproofs}
We begin by proving the results for measures. 
\begin{proof}[Proof of Theorem~\ref{thm: measure main}]
We first prove this in the special case where $\mu=\nu$. 
Consider the function $L\colon x\mapsto \log x$. 
Since $\mu$ is AD-regular, we can use Theorem~\ref{thm: image fourier decay} with $\kappa_* = \kappa_2 = d_{\infty}$ and obtain that
\[
|\widehat{\mu_L}(\xi)|\ll |\xi|^{-\sigma}
\]
for 
\[
\sigma=\frac{\kappa_2-1/2}{2+\kappa_2-1/2} - \varepsilon.
\]
Again, because $\mu$ is AD-regular, we see that $\mu_L$ is AD-regular with the same exponent. Therefore by Lemma~\ref{lma: correlationdim}, 
\[
\int |\widehat{\mu_L}(\xi)|^2 \frac{1}{|\xi|^{\rho}}d\xi<\infty
\]
whenever $1 > \rho>1-\kappa_2$. This implies that as long as
\begin{equation}\label{eqn:condition}
2\sigma+\kappa_2>1,
\end{equation}
if we fix $0 < \eta < 2 \sigma + \kappa_2 - 1$ then 
\begin{equation*}
\int |\widehat{\mu_L * \mu_L}|^2 |\xi|^{\eta} d\xi = \int |\widehat{\mu_L}|^4 |\xi|^{\eta} d\xi \ll \int |\widehat{\mu_L}|^2 |\xi|^{\eta - 2\sigma} d\xi < \infty.
\end{equation*}
Thus the measure $\mu_L * \mu_L$ has an $L^2$-density function with fractional derivatives in $L^2$. 
We can then apply the map $e\colon x\mapsto e^x$ to conclude the corresponding result for $\mu \cdot \mu$. 
The threshold $(\sqrt{65}-5)/4$ sits precisely on the boundary of the condition~\eqref{eqn:condition}.

For asymmetric results (i.e. $\mu \neq \nu$), notice what we need is to replace~\eqref{eqn:condition} with 
\[
2\sigma(\mu)+\kappa_2(\nu)>1,
\]
and that $\kappa_2(\nu) > 1/2$. 
From here, we see that a sufficient condition is
\[
\kappa_2(\mu)\kappa_2(\nu)+\kappa_2(\mu)+1.5\kappa_2(\nu)-2.5>0.
\]
Of course, we can switch the roles of $\mu$ and $\nu$ to get another sufficient condition. 
This proves the two-sets part of Theorem~\ref{thm: measure main}. 

For the three-sets part, consider the convolution of three measures 
\[\chi \coloneqq \mu_L * \nu_L * \gamma_L. \]
    We can now consider its Fourier transform,
    \[
    \widehat{\chi}(\xi) = \widehat{\mu_L}(\xi)\widehat{\nu_L}(\xi)\widehat{\gamma_L}(\xi).
    \]
    We can use Cauchy--Schwarz for two of the three factors and then use the Fourier decay for the last factor\footnote{This three factors ``curse'' is prevalent in modern analytic/additive number theory. There, a lot of results are stated with three or more variables and it is extremely difficult to push them down to two variables, see \cite[Page~192, footnote]{TV}.} in the following integral for a large number $N>0$. 
    Indeed, our assumptions imply that $\kappa_2(\gamma) > 1/2$, so for all $\varepsilon > 0$, 
    \begin{align*}
    \int_{N/2\leq |\xi|\leq N} |\widehat{\chi}(\xi)|d\xi &\ll N^{\varepsilon/3 -\sigma}\left(\int_{|\xi|\leq N} |\widehat{\mu_L}(\xi)|^2d\xi \int_{|\xi|\leq N} |\widehat{\nu_L}(\xi)|^2d\xi\right)^{1/2} \\
    &\ll N^{\varepsilon-\sigma} N^{1-(\kappa_2 (\mu_L)+\kappa_2(\nu_L))/2},
    \end{align*}
    where $\sigma>0$ is related to $\kappa_2(\gamma)$ as in Theorem~\ref{thm: image fourier decay}, namely $\sigma = \frac{\kappa_2(\gamma) - 0.5}{\kappa_2(\gamma) + 1.5}$. 
    We have used the fact that
    \begin{align*}
    \int_{|\xi|\leq N} |\widehat{\mu_L}(\xi)|^2d\xi \ll N^{1-\kappa_2(\mu_L) + \varepsilon/3}
    \end{align*}
    (and similarly for $\nu$) by the definition of $\kappa_2$ and because $\mu,\nu$ are AD-regular. 
    We see from above that as long as
    \begin{align}\label{con: three}
    \sigma+\frac{\kappa_2(\mu_L)+\kappa_2(\nu)}{2}>1,
    \end{align}
    we have
    \[
    \int_{\RR}|\widehat{\chi}(\xi)|d\xi<\infty,
    \]
    so by Lemma~\ref{lma: abscts}, $\chi$ has a continuous density. Therefore the same must be true for $\mu \cdot \nu \cdot \gamma$. 
    Finally, using the fact that $L$ is smooth, condition~\eqref{con: three} becomes~\eqref{e:threemeascondition}, as required. 
\end{proof}

\begin{proof}[Proof of Proposition~\ref{prop: escarith}]
We know that $\kappa_2$ is preserved under the diffeomorphism $L$. 
We can trivially bound $\kappa_*\leq 1$ to get a Fourier decay exponent $\sigma$ arbitrarily close to $(2\kappa_2(\mu)-1)/5$ from Theorem~\ref{thm: image fourier decay}. Now we follow the proof of Theorem~\ref{thm: measure main}. 
Condition~\eqref{eqn:condition} is 
\[
\frac{2(2\kappa_2(\mu) - 1)}{5} + \kappa_2(\nu) > 1, 
\]
which becomes $4\kappa_2(\mu) + 5\kappa_2(\nu) > 7$. 
We can of course interchange the roles of $\mu,\nu$. This gives the second bullet point. 
We see that this condition is satisfied if $\min\{\kappa_2(\mu),\kappa_2(\nu)\} > 7/9$. 

Finally, if $\nu$ is AD-regular, then we know $\kappa_2(\nu) > 1/2$ and we have $\sigma$ arbitrarily close to $\frac{\kappa_2(\nu) - 1/2}{\kappa_2(\nu) + 3/2}$ for $\nu$, so $2\sigma + \kappa_2(\mu)>1$ rearranges to the desired condition. 
\end{proof}

We can use the results for measures to prove the results for sets. 
We use the following standard lemma. 
\begin{lma}\label{lma: innerapproximation}
    Let $F$ be a self-similar set on $\RR$. 
    Then for all $\varepsilon > 0$ there exists a homogeneous self-similar set $F' \subset F$ with the SSC, such that either $F' \subset (-\infty,0)$ or $F' \subset (0,\infty)$, and such that 
    \[
    \Haus F'\geq \Haus F-\epsilon.
    \]
\end{lma}
\begin{proof}
    Let $\Lambda$ be a self-similar IFS generating $F$. 
    We choose $N \in \NN$ large and select a sub-IFS of the iterated system $\Lambda^N$ which satisfies the SSC, is homogeneous, maps $F$ into either $(-\infty,0)$ or $(0,\infty)$, and has attractor $F'$ with the largest possible dimension subject to these constraints. 
    It is not difficult to see that if we had chosen $N$ large enough then $\Haus F'\geq \Haus F-\epsilon$. 
\end{proof}

\begin{proof}[Proof of Theorem~\ref{thm: enough variables mul}]
    First extract self-similar subsets as in Lemma~\ref{lma: innerapproximation}, and with $\varepsilon$ small enough so that the subsets satisfy the same dimension constraints as the original sets. 
    Without loss of generality we may assume all subsets are contained in $(0,\infty)$. 
    Since the subsets are homogeneous with the SSC, the uniform self-similar measures supported on them are AD-regular with the same dimension. 
    Applying Theorem~\ref{thm: measure main} to these subsets completes the proof, noting that the support of an absolutely continuous measure has positive Lebesgue measure, and if the density is continuous then the support has non-empty interior. 
\end{proof}

\begin{proof}[Proof of Theorem~\ref{thm: radial}]
    First note that $E$ and $F$ contain affine copies $E', F'$ of themselves such that $E' \times F' \subseteq (E \times F) \setminus (\{x = a\} \cup \{y=b\})$. 
    Note that $\dim_{\mathrm H} E' = \dim_{\mathrm H} E$ and $\dim_{\mathrm H} F' = \dim_{\mathrm H} F$. 
    We may assume without loss of generality that $E' \times F' \subset (a,\infty) \times (b,\infty)$. 
    Now one can show that the Lebesgue measure of $R_{(a,b)}(E' \times F')$ is positive as in the above proofs, by considering the convolution of the image of an AD-regular self-similar measure on $E'$ under $x\mapsto \log(x-a)$ with the image of an AD-regular self-similar measure on $F'$ under $y \mapsto -\log (y-b)$, both having $\kappa_2 > 1/2$. 
    We omit the details. 
\end{proof}

\begin{proof}[Proof of Theorem~\ref{thm: higher to one}]
Since $s > 2+(k/2)$, we can fix $\sigma$ such that 
\[ \frac{1}{2} < \sigma < \frac{2s-k}{4+2s-k}. \] 
Noting from Lemma~\ref{lma: l2dimHausdorff} that $\kappa_2(\mu) = s$, by Theorem~\ref{thm: image fourier decay} we have $|\widehat{\mu_f}(\xi)| \ll |\xi|^{-\sigma}$, so $\int_{\RR} |\widehat{\mu_f}(\xi)|^2 d\xi < \infty$. 
Therefore by Lemma~\ref{lma: abscts}, $\mu_f$ has an $L^2$-density. 

For the conclusion regarding self-similar sets, it is enough to find an $s$-AD-regular homogeneous self-similar measure on $E$ with $s > 2+(k/2)$ using the non-expanding property, as in Lemma~\ref{lma: innerapproximation}. 
\end{proof}

\section{Further developments and open questions}\label{sec: further dev}

\subsection{Fractal uncertainty principle}\label{ss:fup}

To prove the $l^2$ estimate for Theorem~\ref{thm: image fourier decay} when $\mu$ is homogeneous and rotation-free with $r_n = 2^{-n}$, after applying Cauchy--Schwarz, equation~\eqref{i:useclaim} gave that for all $\varepsilon > 0$, 
\begin{align}\label{e: aftercauchyschwarz}
|\widehat{\mu_f}(\xi)| &\ll 2^{(-\kappa_2' + (2-\gamma)\kappa_*' + \varepsilon k )n/2} \left( \int_{\mbox{dist}(\bxi',2^{-n}|\bxi| P(\supp(\mu))) \leq 2^{\varepsilon n}} |\widehat{\mu}(\bxi')|^2 d\bxi' \right)^{1/2}\\*
&+ 2^{-(2- \gamma)n} \notag
\end{align}
We then crudely bounded the integral in terms of $\int_{B_{C|\bxi|/2^n}(\mathbf{0})} |\widehat{\mu}(\bxi')|^2 d\bxi'$ for a constant $C$. 
If we impose some additional porosity assumptions on $\mu$, however, then one can use the fractal uncertainty principle (FUP) to improve the exponent of decay. 
Indeed, in~\eqref{e: aftercauchyschwarz} the integral of the square of the Fourier transform of the fractal measure $\mu$ is taken over a neighbourhood of another fractal set, namely some distorted and scaled-up copy of $\supp(\mu)$. 
The FUP says, roughly speaking, that no function can be localised in both position and frequency near a fractal set. 
The goal of this section is not to give the most general possible statements; rather, we hope to illustrate how FUP can be used as a tool in this context.

We now formulate the FUP more precisely; a good survey of the FUP is given in~\cite{DyatlovFUPsurvey}. 
For $h \in (0,1)$ we say that an $h$-dependent family of compact subsets $X_h$ and $Y_h$ of $\RR^k$ satisfy an FUP with exponent $\beta \geq 0$ and constant $C \geq 0$ if for all $f \in L^2(\RR^k)$, 
\begin{equation}\label{e: fupdefine}
\{ \bxi \in \RR^k : \widehat{f}(\bxi) \neq 0 \} \subseteq Y_h \qquad \Rightarrow \qquad ||f||_{L^2(X_h)} \leq C h^{\beta} ||f||_{L^2(\RR^k)}.
\end{equation}
Here, the $L^2$ norms are always taken with respect to $k$-dimensional Lebesgue measure, and crucially, $\beta$ and $C$ must be independent of $h$ and $f$. 
If we fix $\varepsilon > 0$ and compact sets $X,Y \subset \RR^k$ and take $Y_h$ to be the $1$-neighbourhood of $h^{-1}Y \coloneqq \{ h^{-1} \by : \by \in Y\}$ and $X_h$ to be the $h$-neighbourhood of $X$, then a straightforward volume argument shows that one can take 
\[
\beta = \max\Big\{\frac{1}{2}(k - \overline{\dim}_{\mathrm B} X - \overline{\dim}_{\mathrm B} Y - \varepsilon),0\Big\}; 
\]
one is interested in when $\beta$ can be larger than this. 

The following lemma translates the FUP into a form where it can be applied to~\eqref{e: aftercauchyschwarz}. 
We work in the setting of Lemma~\ref{lma: homogquant} and let $P$ be the map from the proof of Claim~\ref{claim: assouad}. 
\begin{lma}\label{lma: fup}
    Assume that for $h \in (0,1)$, the sets $X_h = -P(\supp(\mu))^{h}$ and $Y_h = (h^{-1}\supp(\mu))^1$ satisfy a FUP with exponent $\beta > 0$. 
    Then for all $\varepsilon \in (0,1)$,
    \begin{align*}
    \int_{(h^{-1/(1+\varepsilon)}P(\supp(\mu)))^{h^{-\varepsilon}}} |\widehat{\mu}(\bxi')|^2 d\bxi' &\ll h^{2\beta/(1+\varepsilon) - k\varepsilon} \int_{|\bxi'| \ll h^{-(1+\varepsilon)}} |\widehat{\mu}(\bxi')|^2 d\bxi' \\
    &\ll h^{2\beta - (k-\kappa_2) - 3 (\beta + k) \varepsilon}.
    \end{align*}
\end{lma}
\begin{proof}
    The final inequality is from the definition of $\kappa_2$, so we prove the first inequality. 
    Let $\phi \colon \RR^k \to \RR$ be a non-negative smooth function with compact support in $B_1(\mathbf{0})$ and $\int_{\RR^k} \phi(\bx) d\bx = 1$. 
    For $\lambda \in (0,1]$ let $\phi_{\lambda}(\bx) = \lambda^{-k} \phi(\lambda^{-1} \bx)$, noting that $\int_{\RR^k} \phi_{\lambda}(\bx) d\bx = 1$ by the change of variable formula. 
    By Fubini's theorem, $\phi_{\lambda} * \mu$ is a measure supported in $(\supp(\mu))^{\lambda}$ with $||\phi_{\lambda} * \mu||_1 \leq ||\phi_{\lambda}||_1 \mu(\RR^k) = 1$. 
    Let $S_{\lambda}$ be the scaling map defined by $S_{\lambda}(\psi)(\bx) = \lambda^{-k/2} \psi(\bx / \lambda)$. 

    Now, $\supp(S_{1/h} (\phi_{h} * \mu)) \subseteq Y_h$. 
    Therefore the FUP gives 
    \[
    || \widecheck{S_{1/h} (\phi_{h} * \mu)} ||_{L^2(X_h)} \ll h^{\beta} ||  \widecheck{S_{1/h} (\phi_{h} * \mu)}  ||_{L^2(\RR^k)},
    \]
    where $\widecheck{f}(\bx) = \widehat{f}(-\bx)$ is the inverse Fourier transform of an $L^2$ function $f$. 
    By scaling properties of Fourier transforms, this implies 
    \[
    ||S_{h} (\widehat{\phi_{h}} \cdot \widehat{\mu}) ||_{L^2(-X_h)} \ll h^{\beta} ||S_{h} (\widehat{\phi_{h}} \cdot \widehat{\mu}) ||_{L^2(\RR^k)}.
    \]
    By the change of variables formula, $S_{\lambda}$ preserves the $L^2$ norm, so 
    \[
    ||\widehat{\phi_{h}} \cdot \widehat{\mu}||_{L^2(-X_h/h)} \ll h^{\beta} ||\widehat{\phi_{h}} \cdot \widehat{\mu}||_{L^2(\RR^k)}.
    \]
    But $|\widehat{\phi_h}(\bxi)| \asymp 1$ for $|\bxi| \leq h^{-1/(1+0.1\varepsilon)}$ so $||\widehat{\phi_{h}} \cdot \widehat{\mu}||_{L^2(-X_h/h)} \asymp ||\widehat{\mu}||_{L^2(-X_h/h)}$. 
    Moreover, $\widehat{\phi_{h}}(\bxi) = \widehat{\phi}(h \bxi)$, so since $\widehat{\phi}$ is a Schwartz function decaying faster than any polynomial, for all $L > 0$ we have $||\widehat{\phi_{h}}||_{L^2(|\bxi| \geq h^{-(1+\varepsilon)})} \ll h^{\varepsilon L}$. 
    Therefore setting $L=100k/\varepsilon$, 
    \begin{align*}
    ||\widehat{\mu}||_{L^2(-X_h/h)} &\ll h^{\beta} ||\widehat{\phi_{h}} \cdot \widehat{\mu}||_{L^2(\RR^k)} \\
    &\ll h^\beta ( ||\widehat{\phi_{h}}||_{L^2(\RR^k \setminus B_{h^{-(1+\varepsilon)}}(\mathbf{0}))}  +  ||\widehat{\mu}||_{L_2(B_{h^{-(1+\varepsilon)}}(\mathbf{0}))}) \\
    &\ll h^{\beta} ( h^{100k} + ||\widehat{\mu}||_{L_2(B_{h^{-(1+\varepsilon)}}(\mathbf{0}))}) \\
    &\ll h^{\beta} ||\widehat{\mu}||_{L_2(B_{h^{-(1+\varepsilon)}}(\mathbf{0}))}.
    \end{align*}
    
    Similarly, applying the FUP at scale $h^{1/(1+\varepsilon)}$, since translating a function preserves its Fourier support, uniformly in translations $\mathbf{v}$ we have \[  \int_{(h^{-1/(1+\varepsilon)} P(\supp(\mu)) )^1 + \mathbf{v}} |\widehat{\mu}(\bxi')|^2 d \bxi' \ll h^{2\beta /(1+\varepsilon)} \int_{|\bxi'| \ll h^{-(1+\varepsilon)}} |\widehat{\mu}(\bxi')|^2 d\bxi'. \] The result follows.
\end{proof}

We can use a FUP which was recently proved by Cohen~\cite{CohenFUPannals} to get an improved exponent. 
A compact set $F \subset \RR^k$ (or a measure whose support is $F$) is called porous from scale $\alpha_0$ to $\alpha_1$ if there exists some $\nu \in (0,1/3)$ such that for every ball $B$ of radius $R \in (\alpha_0,\alpha_1)$ there is some $\bx \in B$ such that $B_{\nu R}(\bx) \cap F = \varnothing$. 
It is called line-porous from scale $\alpha_0$ to $\alpha_1$ if for some $\nu \in (0,1/3)$, for all line segments $\tau$ of length $R \in (\alpha_0,\alpha_1)$, there is some $\bx \in \tau$ such that $B_{\nu R}(\bx) \cap F = \varnothing$. 
We say a measure is porous (respectively line-porous) from scale $\alpha_0$ to $\alpha_1$ if its support is porous (resp. line-porous) from scale $\alpha_0$ to $\alpha_1$. 
Line-porosity is a strictly stronger condition than porosity when $k \geq 2$; for example, the Sierpi\'nski carpet is porous from scale $0$ to $1$ but not line-porous. 
We will present the proofs of the following results assuming $\mu$ is homogeneous and rotation-free with $r_n = 2^{-n}$ for simplicity, but the extension to the general case is very similar to Section~\ref{ss: nonhomog}. 
\begin{thm}\label{thm: largedimfup}
    Let $\mu$ be a non-expanding self-similar measure on $\RR^k$ which is line-porous from scale $0$ to $1$ and satisfies $\kappa_2 > k/2$. 
    Let $f \colon \RR^k \to \RR$ be a $C^2$ map whose graph has non-vanishing Gaussian curvature over $\supp(\mu)$. 
    Then there exists $\upsilon > 0$ such that $|\widehat{\mu_f}(\xi)|\ll |\xi|^{-\sigma}$ with 
    \[ \sigma \coloneqq \frac{2\kappa_2 - k}{4 + 2\kappa_* - k} + \upsilon. \] 
\end{thm} 
\begin{proof}
    Let $X_h,Y_h$ be as in Lemma~\ref{lma: fup}. 
    Note that $Y_h$ is line-porous from scale $1$ to $h^{-1}$. 
    Moreover, $P(\supp(\mu))$ can be covered by a finite union of diffeomorphic copies of $\supp(\mu)$ by compactness, so $X_h$ is porous from scale approximately $h$ to $1$. 
    Therefore by \cite[Theorem~1.1]{CohenFUPannals}, $X_h$ and $Y_h$ satisfy a FUP with some exponent $\beta > 0$. 
    Applying Lemma~\ref{lma: fup} with $h = (|\bxi|/2^n)^{-(1+\varepsilon)} = 2^{-(1+\varepsilon)(\gamma - 1)n}$ to~\eqref{e: aftercauchyschwarz}, we see that for all $\varepsilon \in (0,1)$, 
    \[
    |\widehat{\mu_f}(\xi)| \ll 2^{(-d_2' + (2-\gamma)\kappa_*' - 2(\gamma - 1) \beta +  (\gamma - 1)(k-\kappa_2) + O(\varepsilon) )n/2} + 2^{-(2- \gamma)n}. 
    \]
    Making $\varepsilon > 0$ very small compared to $\beta$, we see that the value of $\gamma$ for which the two expressions are equal is strictly smaller than in the proof of Theorem~\ref{thm: image fourier decay} (the size of the improvement depends on $\beta$). 
    We can therefore make $\upsilon$ small enough in terms of $\beta$ so that the desired Fourier decay holds. 
\end{proof}

In the $k=1$ case there is a nice characterisation of line-porosity. 
\begin{prop}[\cite{Luukkainen1998,FraserAssouadSelfSim}]
    Let $F$ be a self-similar set in $\RR$. The following are equivalent: 
    \begin{enumerate}
        \item\label{i:lpor} $F$ is line-porous,
        \item\label{i:porous} $F$ is porous,
        \item\label{i:assdim} $\dim_{\mathrm A} F < 1$,
        \item\label{i:wsc} $\dim_{\mathrm H} F < 1$ and $F$ satisfies the weak separation condition. 
    \end{enumerate}
\end{prop}
\begin{proof}
    \eqref{i:lpor} $\Leftrightarrow$ \eqref{i:porous} is obvious. 
    We have \eqref{i:porous} $\Leftrightarrow$ \eqref{i:assdim} by \cite[Theorem~5.1.6]{Fraser2020book}. We have \eqref{i:assdim} $\Leftrightarrow$ \eqref{i:wsc} by \cite[Theorem~7.2.4]{Fraser2020book}. 
\end{proof}

\begin{rem}
    The FUP can be used to improve some of the nonlinear arithmetic results from Section~\ref{sec: arithmetic}. 
    For instance, by Theorem~\ref{thm: largedimfup}, the assumption~\eqref{e:twomeascondition} in the statement of Theorem~\ref{thm: measure main} can be weakened to the same assumption with $\geq$ rather than $>$. 
    Moreover, if we assume our self-similar set $F \subset \RR$ is AD-regular then the ``$\Haus F >$'' conditions in the abstract can be replaced by the corresponding conditions with ``$\Haus F \geq$''. 
\end{rem}

We now turn to the proofs of Theorems~\ref{thm: cladektao} and~\ref{thm: blt} from Section~\ref{sec: large linear space}. 
We say that an $h$-dependent family of sets $X_h$ is $\delta$-regular on scales $[\alpha_0,\alpha_1]$ if it supports some locally finite measure $\nu$ for which there exists $C\geq 1$ (independent of $h$) such that for all balls $B$ centred at a point in $X_h$ with radius $r \in (\alpha_0,\alpha_1)$ we have $C^{-1}r^{\delta} \leq \nu(B) \leq C r^{\delta}$. 
\begin{proof}[Proof of Theorem~\ref{thm: blt}]
    The $s > k/2$ case of Theorem~\ref{thm: blt} follows immediately from Theorem~\ref{thm: image fourier decay}, so we assume $s \leq k/2$. 
    The sets $X_h$ and $Y_h$ from Lemma~\ref{lma: fup} are $s$-regular on scales $[h,1]$ and $[1,h^{-1}]$ respectively. 
    Since the support of $\mu$ is not contained in any proper affine subspace, the non-orthogonality hypothesis from \cite[Theorem~1.5]{BLTfup} holds. 
    Therefore there exists $\beta > 0$ such that $X_h,Y_h$ satisfy a FUP\footnote{This paper uses a slightly different formulation of the FUP than~\eqref{e: fupdefine}, but it is not difficult to show the two are equivalent.} with exponent $k/2 - s + \beta$. 
    In this case,~\eqref{e: aftercauchyschwarz} becomes that for all $\varepsilon > 0$, 
    \begin{equation*}
    |\widehat{\mu_f}(\xi)| \ll 2^{(-\kappa_2 + (2-\gamma)\kappa_* + (\gamma - 1)(k-\kappa_2) - (\gamma - 1)(k - 2 s + 2\beta) + \varepsilon)n/2} + 2^{-(2- \gamma)n}.
    \end{equation*}
    But $\kappa_* = \overline{\dim}_{\mathrm B} (\supp(\mu)) = \kappa_2 = s$, so this reduces to 
    \begin{equation}\label{e:simplifiedafterfup}
    |\widehat{\mu_f}(\xi)| \ll 2^{(-(\gamma - 1)\beta + \varepsilon/2) n}  + 2^{-(2- \gamma)n}.
    \end{equation}
    Since $\varepsilon$ can be made arbitrarily small relative to $\beta$, letting $\gamma = (2+\beta)/(1+\beta)$, we see that the exponent of Fourier decay for $\mu_f$ can be made arbitrarily close to $\beta / (2+\beta)$. 
\end{proof}

\begin{proof}[Proof of Theorem~\ref{thm: cladektao}]
    The sets $X_h$ and $Y_h$ from Lemma~\ref{lma: fup} are $k/2$-regular on scales $[h,1]$ and $[1,h^{-1}]$ respectively. 
    Therefore \cite[Theorem~1.9]{CladekTaoFUP} implies that $X_h,Y_h$ satisfy a FUP with some exponent $\beta > 0$. 
    Now as in~\eqref{e:simplifiedafterfup}, Lemma~\ref{lma: fup} means that~\eqref{e: aftercauchyschwarz} reduces to 
    \[
    |\widehat{\mu_f}(\xi)| \ll 2^{(\varepsilon - (\gamma - 1)\beta)n} + 2^{-(2- \gamma)n}
    \]
    for $\epsilon > 0$ which can be made arbitrarily small. 
    Letting $\gamma = (2+\beta)/(1+\beta)$, the exponent of Fourier decay can be made arbitrarily close to $\beta / (2+\beta)$. 
\end{proof}

There has been much interest in estimating the exponent $\beta$ in the FUP. 
For instance, for $s$-AD-regular sets in the line with $1 > s \geq 1/2$ and multiplicative AD-regularity constant $C$, Cohen's result was proved in an earlier celebrated paper of Bourgain and Dyatlov~\cite{BourgainDyatlovFUP}. The dependence of the exponent on $s$ and $C$ was described in~\cite{JinZhangFUP,JZZ}. 
Estimates for the exponents are also given in~\cite{CladekTaoFUP,DJfupDolgopyat}. 
For the special case of missing digit sets in the line, there has been yet more work~\cite{EswarathasanHan,LaiShi}. 
These estimates of course translate to the exponent of Fourier decay for the pushforward measure. In particular, while Theorem~\ref{thm: blt} was known in the $k=1$ case~\cite{BB25,ACWW25}, its proof gives new quantification of the exponent (\cite{MS18} also gives quantification but only applies when the IFS is homogeneous). A consequence of the quantifications from~\cite{JinZhangFUP,DJfupDolgopyat} is the following. 
\begin{rem}
    At least in the $k=1$ case of Theorem~\ref{thm: largedimfup} when $\mu$ is AD-regular, the improvement $\upsilon$ can be taken to depend only on $\mu$ but not on $f$. 
    Similarly, in the $k=1$ case of Theorem~\ref{thm: blt}, the decay exponent depends only on $\mu$ but not on $f$. 
    Indeed, iterate the IFS to some large fixed level $N$ so that for each copy $\mu^{(\omega)}$ and all $x,y$ in the convex hull of $\supp(\mu^{(\omega)})$ we have $|f''(x)|\leq 1.01 |f''(y)|$. Then on scales $(0,\mbox{diam}(\supp(\mu^{(\omega)})))$, $r_{\omega}^s \mu^{(\omega)}$ is AD-regular with the same exponent as $\mu$ and at most twice the constant as for $\mu$. Therefore the gain in exponent for the decay of $|\widehat{\mu^{(\omega)}_f}|$ from~\cite{JinZhangFUP} and~\cite{DJfupDolgopyat} respectively depends only on $\mu$. 
    Thus the gain in exponent for the decay of $|\widehat{\mu_f}|$ depends only on $\mu$ (the multiplicative constant of course also depends on $f$). 
\end{rem}

\subsection{Quadratic maps}
Theorem~\ref{thm: image fourier decay} handles maps $\mathbb{R}^k\to\mathbb{R}$. However, our method extends beyond this case. We illustrate some concrete examples.

Let $f\colon \mathbb{R}^k\to\mathbb{R}^d$ be a quadratic map. Namely, $f=(f_1,\dots,f_d)$ where $f_1,\dots,f_d$ are all quadratic polynomials with $k$ variables. More specifically, for each $i\in\{1,\dots,d\}$, there are real numbers $c_{i,p,q}$ where $1 \leq p \leq q \leq k$ such that
\[
f_i(x_1,\dots,x_k)=\sum_{p,q} c_{i,p,q} x_px_q+\text{affine part}.
\]
Here the affine part is an affine map. 

Let $\bv\in\mathbb{S}^{d-1}$ be a vector. Consider the following projection of $f$ along the direction $\bv$: 
\[
f_\bv=\sum_{i=1}^d v_i f_i.
\]
For each fixed $\bv$, the map $f_\bv \colon \mathbb{R}^k\to\mathbb{R}$ can be treated with the method that we have developed. First consider the Gaussian curvature of the graph of $f_\bv$. 
In order to do this, we compute the Hessian matrix $H^{f,\bv}$ of $f_\bv$. It is a symmetric $k\times k$ matrix whose $(p,q)$ entry is
\[
H^{f,\bv}_{pq}=\frac{\partial}{\partial x_p}\frac{\partial}{\partial x_q} f_\bv=\sum_{i=1}^d c_{i,p,q}v_i \frac{\partial}{\partial x_p}\frac{\partial}{\partial x_q} x_p x_q.
\]
As long as $\bv$ is fixed, it is clear that the determinant $\det H^{f,\bv}$ is a constant (either zero or non-zero). 
Supposing that this determinant is non-zero, we can apply Theorem~\ref{thm: image fourier decay} to get Fourier decay for $\mu_{f_\bv}$, namely 
\begin{equation}\label{e:quadraticonedirectiondecay}
|\widehat{\mu_{f_\bv}}(\xi)|\ll_\bv |\xi|^{-\sigma},
\end{equation}
where $\sigma$ and $\mu$ are as in Theorem~\ref{thm: image fourier decay}. 
Here the constant implicit in $\ll_\bv$ depends on the choice of $\bv$, but if $\det H^{f,\bv}$ is non-zero for all $\bv$ then we can remove the dependence on $\bv$ as described in the following result. 

\begin{cor}
    Let $f \colon \RR^k \to \RR^d$ be a quadratic map with $\det H^{f,\bv} \neq 0$ for all $\bv$, let $\mu$ be a non-expanding self-similar measure on $\RR^k$ with $\kappa_2 > k/2$, and let $\sigma > 0$ be as in the statement of Theorem~\ref{thm: image fourier decay}. Then for $\bxi \in \RR^d \setminus \{\mathbf{0}\}$, 
    \[
    |\widehat{\mu_f}(\bxi)|\ll |\bxi|^{-\sigma}.
    \]
\end{cor}
\begin{proof}
From the proof of Theorem~\ref{thm: image fourier decay} (recall in particular Lemma~\ref{lma: bounded overlaps} giving separation of tangents), the implicit constant in~\eqref{e:quadraticonedirectiondecay} depends continuously on $H^{f,\bv}$, which in turn varies smoothly with $\bv$. 
Therefore since $\det H^{f,\bv}\neq 0$ for all $\bv$ and $\mathbb{S}^{d-1}$ is compact, we can improve the Fourier decay property to $|\widehat{\mu_{f_\bv}}(\xi)|\ll |\xi|^{-\sigma}$
where $\ll$ does not depend on $\bv$. 
But for each $\bxi \in \RR^d \setminus \{\mathbf{0}\}$, 
\[
\widehat{\mu_{f}}(\bxi) = \widehat{\mu_{f_{\bxi/|\bxi|}}}(|\bxi|),
\]
completing the proof. 
\end{proof}

\begin{exm}\label{exm: quad analytic}
    Let $f\colon \mathbb{R}^2\to\mathbb{R}^2$ be given by 
    \[
    f(x,y) = (x^2-y^2,2xy).
    \]
    We see that for $\bv\in\mathbb{S}^1$, 
    \[
    \det H^{f,\bv}=-4v^2_1-4v^2_2 = -4.
    \]
    Then for any self-similar measure $\mu$ on $\mathbb{R}^2$ with $\kappa_2 > 1$, the measure $\mu_f$ on $\mathbb{R}^2$ has polynomial Fourier decay with exponent $\sigma$ indicated in Theorem~\ref{thm: image fourier decay}. 
\end{exm}

\begin{exm}
    Let $f \colon \mathbb{R}^k\to\mathbb{R}^k$ be
    \[
    f(x_1,\dotsc,x_k) = (x_1^2,\dotsc,x_k^2). 
    \]
    For $\bv\in\mathbb{S}^{k-1}$, 
    \[
    \det H^{f,\bv}= 2^k \prod_{i=1}^k v_i.
    \]
    Let $\mu$ be a non-expanding self-similar measure on $\RR^k$ with $\kappa_2 > k/2$. 
    Then for all $\varepsilon > 0$, for directions $\bv$ with $\min\{|v_1|,\dots,|v_k|\}>\varepsilon$, we have 
    \[
    |\widehat{\mu_{f}}(t\bv)|\ll_{\varepsilon} |t|^{-\sigma}
    \]
    where $\sigma$ is again indicated in Theorem~\ref{thm: image fourier decay} and the implicit constant depends on $\varepsilon$ but not $\bv$. 
\end{exm}

\subsection{Holomorphic maps}
One might expect that images of self-similar measures under holomorphic (complex analytic) maps $\mathbb{C}\to\mathbb{C}$ should have good polynomial Fourier decay properties. Indeed, Example~\ref{exm: quad analytic} gives a hint of this, since the map in this example is simply $z \mapsto z^2$. 
Let $D \subset \CC$ be non-empty and open and let $f \colon D \to \CC$ be holomorphic. Write  
\[
f(x,y)=U(x,y)+i V(x,y),
\]
and for $\bv=(v_1,v_2)\in\mathbb{S}^1$ write $f_\bv \coloneqq v_1U+v_2V$. 
The following lemma gives further indication for why holomorphic pushforwards are easier to work with than, say, real-analytic ones. 
\begin{lma}\label{lma: complexproperties}
    The Hessian matrix $H^{f,\bv}$ admits an orthonormal eigenbasis with eigenvalues $\pm |f''|$. Moreover, 
    \begin{gather*}
    |\det H^{f,\bv}|=|(U_{xy},V_{xy})|^2_2 = |f''|^2, \\
    \{\det H^{f,\bv}=0\} = \{U_{xy}=V_{xy}=0\} = \{U_{xx}=V_{xx}=0\} = \{f''=0\}. 
    \end{gather*}
    In particular, these quantities are independent of $\bv$. 
\end{lma}
\begin{proof}
Using the Cauchy--Riemann equations $U_x = V_y$ and $U_y = -V_x$ (where we write $U_x$ for $\frac{\partial U}{\partial x}$ etc.), 
\begin{equation}\label{e:complexhessian}
H^{f,\bv}(x,y) = \begin{pmatrix} v_1 V_{xy} - v_2 U_{xy} & v_1 U_{xy} + v_2 V_{xy} \\ v_1 U_{xy} + v_2 V_{xy} & -v_1 V_{xy} + v_2 U_{xy} \end{pmatrix}.
\end{equation}
Therefore 
\[
\det H^{f,\bv}(x,y)=-(v_1V_{xy}-v_2U_{xy})^2-(v_1U_{xy}+v_2V_{xy})^2.
\]
This gives $|\det H^{f,\bv}|=|(U_{xy},V_{xy})|^2_2$ and $\{\det H^{f,\bv}=0\} = \{U_{xy}=V_{xy}=0\}$. 
The eigenvalues of $H^{f,\bv}(x,y)$ can be obtained from~\eqref{e:complexhessian} with a direct calculation. 
The equality $\{U_{xy}=V_{xy}=0\} = \{U_{xx}=V_{xx} = 0 \}$ comes from Cauchy--Riemann again. 
By the relation 
\[
\frac{df}{dz}=\frac{1}{2}\left(\frac{\partial}{\partial x}-i\frac{\partial}{\partial y}\right)(U+iV),
\]
we see that
\[
\frac{d^2f}{dz^2} = \frac{1}{4}\left(\frac{\partial^2}{\partial x^2}-\frac{\partial^2}{\partial y^2}-2i \frac{\partial}{\partial x}\frac{\partial}{\partial y}\right)(U+iV).
\]
This gives the final equalities 
\[ \{\det H^{f,\bv}=0\}=\{f''=0\}, \qquad |(U_{xy},V_{xy})|^2_2 = |f''|^2. \]
\end{proof}
In particular, if $D$ is connected, then $\{\det H^{f,\bv}=0\}$ is the whole of $D$ if and only if $U$ is affine (in which case $V$ and hence also $f$ are affine); otherwise, $\{\det H^{f,\bv}=0\}$ is a discrete set. 

As an example, if $f(z)=z^3$, then we see that
\[
U(x,y)=x^3-3xy^2; \qquad V(x,y)=-(y^3-3x^2y).
\]
Then $U_{xy}=-6y, V_{xy}=6x$ and the only place that $U_{xy}=V_{xy}=0$ is $z=x+iy=0$. 

We can now use Lemma~\ref{lma: complexproperties} to prove the van der Corput type result Theorem~\ref{thm: complexvdc} for holomorphic pushforwards of self-similar measures on $\CC$. 
A related result in the special case when $f''$ never vanishes and the self-similar measure is homogeneous with non-trivial rotations is~\cite[Theorem~3.1]{MosqueraOlivo}. 

\begin{proof}[Proof of Theorem~\ref{thm: complexvdc}]
    Since $f$ is nonlinear, the intersection of $\{\det H^{f,\bv}=0\}\subset\mathbb{C}$ (which is discrete on a neighbourhood of $\supp(\mu)$) with $\supp(\mu)$ (which is compact) is indeed finite. 
    The idea of the proof will be to take into account the fact that the nonlinearity gets weaker as we approach this set. We will bound the part of $\mu$ lying close to this set using the Frostman exponent, and bound the image of the rest of $\mu$ using the methods of proof from Theorem~\ref{thm: image fourier decay}. 

    Let $\eta \coloneqq l-2$ and assume for now that $\eta > 0$. Then all the zeros of $f''$ which intersect $\supp(\mu)$ have order at most $\eta$. 
    For a set $X \subset \mathbb{C}$ recall that $X^{\delta} \coloneqq \bigcup_{z \in X} B_{\delta}(z)$ denotes the $\delta$-neighbourhood of $X$. 
    For some $C>0$ and all small enough $\delta$, in a sufficiently small neighbourhood of $\supp(\mu)$ we have 
    \[
    \{\det H^{f,\bv}=0\}^{C\delta} \supset \{|\det H^{f,\bv}|<\delta^{2\eta} \}, 
    \]
    uniformly for $\bv$. 
    Fix constants $d_{\infty}',\kappa_*',\kappa_p'$ very close to the respective constants as in the proof of Theorem~\ref{thm: allparams}. 
    Then uniformly across $\bv$, 
    \[
    \mu(\{\det H^{f,\bv}=0\}^\delta )\ll \delta^{d_{\infty}'}. 
    \] 
    The upshot is that 
    \begin{equation*}
    \mu(\{|\det H^{f,\bv}|<\delta^2 \}) \ll \delta^{d_{\infty}'/\eta},
    \end{equation*}
    with the constant describing how small $\delta$ must be and the implicit constant uniform across $\bv \in\mathbb{S}^1$. 

    Define the following numbers: 
    \begin{align*}
        \gamma &\coloneqq \frac{2(\kappa_* d_{\infty} + \eta \kappa_* + d_{\infty})}{\kappa_* d_{\infty} + \eta \kappa_* + d_{\infty} \kappa_2}, \\
        \tau &\coloneqq \frac{\kappa_2 - 1}{\kappa_* d_{\infty} + \eta \kappa_* + d_{\infty}}.
    \end{align*}
    Since $\kappa_* \geq \kappa_2 > 1$, we have that $1 < \gamma < 2$ and since also $\eta > 0$ we have $0< \tau < 1/\gamma < 1$. 

    We now assume $\mu$ is homogeneous for simplicity and follow Lemma~\ref{lma: homogquant}. 
    Since $\mu$ is supported in the plane it is automatically non-expanding, and the extension of the proof to the non-homogeneous case follows the proof of Theorem~\ref{thm: allparams} from Lemma~\ref{lma: homogquant}. 
    Fix $\bv \in \mathbb{S}^1$; importantly, all the estimates in $\ll$ and $\gg$ in the following calculation are independent of $\bv$. 
    Consider the function $f_{\bv} \colon U \to \RR$ and the associated function $T=T_\bv$. 
    Letting $T_\bv(z) = (z,f_\bv(z)) \in \RR^3$ and $\bxi \coloneqq (0,0,\xi) \in \RR^3$, Lemma~\ref{lma: lifttograph} gives 
    \begin{equation}\label{e:directionalFT}
    \widehat{\mu_f}(\xi \mathbf{v}) = \widehat{\mu_{f_\bv}}(\xi) = \widehat{\mu_{T_\bv}}(\bxi).
    \end{equation}
    We will estimate this quantity for $\xi \in \RR$ satisfying $|\xi| = 2^{\gamma n}$. 
    Break the sum from~\eqref{eqn: measure decomposition} into two parts, namely those cubes whose centre lies in the $|\bxi|^{-\tau}$-neighbourhood of $\{\det H^{f,\bv}=0\} \cap \supp(\mu)$ (call this $\mathcal{D}_n^{(1)}$) and those which do not (call this $\mathcal{D}_n^{(2)}$). 
    Fix $\varepsilon_0 > 0$. Then there are constants $\varepsilon_1, \varepsilon_2, \dotsc \in (0,\varepsilon_0)$ such that the rest of the calculations in the proof hold. 
    Since $\tau < 1/\gamma$, we can bound  
    \begin{align}\label{e:complexfrostman}
    \begin{split}
    \sum_{D_n \in \mathcal{D}_n^{(1)}}\sum_{\mu_{\omega,n}\sim D_n} |\mu_{\omega,n}| |\widehat{\mu}(r_n O_n(\bxi_{\omega,n}))| &\ll \mu\left( \bigcup_{D_n \in \mathcal{D}_n^{(1)}} 3D_n \right) \\ 
    &\ll |\bxi|^{-\tau d_{\infty} + \varepsilon_1} = 2^{- (\tau d_{\infty} \gamma - \varepsilon_2) n}. 
    \end{split}
    \end{align}

    Next we consider the cubes in $\mathcal{D}_n^{(2)}$. At each point in such a cube, by Lemma~\ref{lma: complexproperties}, the magnitude of each eigenvalue of $H^{f,\bv}$ is $\gg |\bxi|^{-\eta \tau}$. 
    Modifying the proof of Lemma~\ref{lma: bounded overlaps} on separation of tangents, we see that the points $\bxi_{D_n} r_n$ are $|\bxi| 2^{-2n} |\bxi|^{-\eta \tau}$ separated with bounded multiplicity. 
    Therefore in the proof of Claim~\ref{claim: assouad}, the points $\mathbf{\eta}_{D_n}$ are $|\bxi|^{-\eta \tau} 2^{-n}$ separated with bounded multiplicity, and for any unit square $D$, 
    \[
    \# ( \{ \bxi_{D_n} r_n \}_{D_n \in \mathcal{D}_n^{(2)}} \cap D) \ll \left( \frac{2^n/|\bxi|}{|\bxi|^{-\eta \tau} 2^{-n}} \right)^{\kappa_* + \varepsilon_3} = 2^{n \kappa_*(2-\gamma + \gamma \eta \tau + \varepsilon_4)}. 
    \] 
    Therefore the sum over $\mathcal{D}_n^{(2)}$ can be bounded by 
    \begin{equation}\label{e:complexmain} 
    \ll 2^{n(- \kappa_2 + (2-\gamma + \gamma \eta \tau)\kappa_* + (2-\kappa_2) (\gamma - 1) + \varepsilon_5)/2}. 
    \end{equation}
    We also have the error term from the Taylor expansion: 
    \begin{equation}\label{e:complexerror}
    \ll \frac{|\bxi|}{2^{2n}} = 2^{-(2 - \gamma )n}. 
    \end{equation}
    By our choice of $\tau$ and $\gamma$ (and a direct but tedious algebraic manipulation), the three terms on the right hand sides of~\eqref{e:complexfrostman},~\eqref{e:complexmain},~\eqref{e:complexerror} are equal (up to a small $\varepsilon$ factor). 
    The exponent of decay can therefore be made arbitrarily close to $(2-\gamma)/\gamma$ which one can check is the desired exponent $\frac{\kappa_2 - 1}{1 + \kappa_* + \eta \kappa_* / d_{\infty}}$. 

    Finally, if $\eta = 0$, then $f''$ never vanishes on $\supp(\mu)$, and in light of~\eqref{e:directionalFT} the proof is essentially the same as in Theorem~\ref{thm: image fourier decay}. 
\end{proof}

    Of course, one could follow the proof of Theorem~\ref{thm: image fourier decay} to obtain Fourier decay bounds in the setting of Theorem~\ref{thm: complexvdc} involving $\kappa_p$ and $d_q$; for simplicity we have stated only the $p=q=2$ case. 
The deduction of Theorem~\ref{thm: complexvdc} from the proof of Theorem~\ref{thm: image fourier decay} is somewhat similar to the deduction of the result on real-analytic pushforwards $\RR \to \RR$ from the result on pushforwards under maps with non-vanishing second derivative in~\cite{BB25}. 

If we are not concerned with quantifying the decay then there are several places where Theorem~\ref{thm: complexvdc} can be improved. 
In particular, one can ask the following: 
\begin{itemize}
    \item Can the condition $\kappa_2 > 1$ be replaced by some weaker irreducibility assumption that still rules out the obstruction from Section~\ref{sec: large linear space}? 
    \item Is the holomorphicity of $f$ essential? 
\end{itemize}

An affirmative resolution of the following conjecture would address both questions. 
\begin{conj}\label{conj: dream conjecture}
    Let $k,d$ be positive integers. Let $\mu$ be a self-similar measure on $\mathbb{R}^k$ which is not supported in any proper affine hyperplane. 
    Let $U \subset \RR^k$ be a connected open neighbourhood of $\supp(\mu)$ and let $f\colon U \to\mathbb{R}^d$ be real-analytic and non-degenerate in the sense that the graph of $f$ is not contained in any proper affine hyperplane of $\RR^{k+d}$. 
    Then there exists $\sigma>0$ such that
    \[
    |\widehat{\mu_f}(\bxi)|\ll |\bxi|^{-\sigma}.
    \]
\end{conj}
 We address this conjecture in the subsequent article~\cite{BYqualitative} under some assumptions, for example we can prove it under the additional assumption that the self-similar measure is non-expanding. 
 This allows us to prove polynomial Fourier decay for a large class of nonlinear self-conformal measures on the plane. 

We now describe several more challenging open questions relating to the work done in this article. We would be thrilled to see them solved in the future. 
\subsection{Optimal bounds}\label{ss: optimality}
 We believe that the lower bounds in Theorem~\ref{thm: image fourier decay} are not optimal, due to the loss described in Remark~\ref{rem: loss}. 
\begin{ques}\label{q: optimal}
Under the conditions of Theorem~\ref{thm: image fourier decay} together with the additional assumption that the self-similar measure is AD-regular, is it the case that for each $\varepsilon > 0$ one can take $\sigma$ to be as large as $-\varepsilon + \kappa_2(\mu)/2$? 
\end{ques}
An affirmative answer to Question~\ref{q: optimal} would have major implications. It would solve, for instance, the two-sets case of Conjecture~\ref{conj: multiplication} on nonlinear arithmetic of self-similar sets. 
It would also show that, under the conditions of Theorem~\ref{thm: image fourier decay}, pushforwards of self-similar sets with the OSC by smooth maps whose graphs have non-vanishing Gaussian curvature are \emph{Salem sets}, i.e. sets $K$ satisfying $\dim_{\mathrm F} K = \dim_{\mathrm H} K$. Here, the \emph{Fourier dimension} of a set $K \subset \RR^k$ is defined by $\dim_{\mathrm F} K \coloneqq \sup \{\min\{\dim_{\mathrm F} \nu, k\}\}$, where the supremum is taken over finite Borel probability measures $\nu$ supported on $K$ (one always has $\dim_{\mathrm F} K \leq \dim_{\mathrm H} K$). In particular, the image of the middle-third Cantor set (recall Example~\ref{exm: cantor lebesgue}) by $x \mapsto x + x^2$ would be a self-conformal Salem set with Fourier dimension $\frac{\log 2}{\log 3} \approx 0.631$. 
 To date, no IFS attractor with non-integer Hausdorff dimension has been proved to be a Salem set. 
 Sets that have been proved to be Salem include random constructions and sets from Diophantine approximation, see~\cite{Ma2,KaufmanSalem} and \cite[Chapter~17]{K85}. 

 \subsection{Arithmetic products of self-similar sets} 
 We have already posed Conjectures~\ref{conj: multiplication} and~\ref{conj:multconv} for the optimal thresholds for the nonlinear arithmetic questions considered in Section~\ref{sec: arithmetic}. 
 There are various variants of these conjectures which one could pose. 
 If Conjecture~\ref{conj:multconv} were true (even in the special case where the self-similar measures are AD-regular), then Conjecture~\ref{conj: multiplication} would follow; this is a consequence of the approximation argument used in the proofs of the results for self-similar sets in Section~\ref{ss:nonlinearproofs}. 
 We are not aware of any counterexample to the possibility that the conclusion of Conjecture~\ref{conj: multiplication} could be strengthened from positive measure to non-empty interior. 
 We summarise what is and is not known about the two- and three-fold arithmetic product of a self-similar set $F \subset (0,\infty)$: 
 \begin{itemize}
    \item If $\Haus F > 0.765\dotsc$ then $F \cdot F$ has positive Lebesgue measure by Theorem~\ref{thm: arith headline}, but we do not know if this must still be the case when $1/2 < \Haus F \leq 0.765\dotsc$. 
    We do not know any Hausdorff dimension\footnote{There is an $l^1$-dimension threshold, see for example~\cite{YuManifold}. However, compared to the Hausdorff dimension, the corresponding theory for $l^1$-dimension for self-similar systems is far less known.} threshold that guarantees $F \cdot F$ has non-empty interior. 
     \item If $\Haus F > 0.850\dotsc$ then $F\cdot F\cdot F$ has non-empty interior by Theorem~\ref{thm: arith headline}. 
     If $1/2 < \Haus F \leq 0.850\dots$ then $F\cdot F\cdot F$ has positive measure by the proof of Theorem~\ref{thm: enough variables mul} and \cite[Theorem~1.23]{H10} but we do not know if it necessarily has non-empty interior. 
     If $1/3 < \Haus F \leq 1/2$ then we do not know if $F\cdot F\cdot F$ must have positive measure or non-empty interior. 
 \end{itemize}

\section*{Acknowledgements}

\textbf{Thanks.} We thank Simon Baker for some very helpful discussions. We thank Amir Algom for helpful comments on a draft version of this manuscript and for pointing out the reference~\cite{EkstromSchmeling}. We also thank two anonymous reviewers for some helpful comments. 

\textbf{Funding.} Most of this work was completed while AB was based at Loughborough University and supported by Simon Baker's EPSRC New Investigators Award (EP/W003880/1). AB was also supported by Tuomas Orponen's Research Council of Finland grant (no.~355453) at the University of Jyv\"askyl\"a. 
HY was financially supported by the Leverhulme Trust (ECF-2023-186). 

\textbf{Rights.} For the purpose of open access, the authors have applied a Creative Commons Attribution (CC-BY) licence to any Author Accepted Manuscript version arising from this submission.

\textbf{AI statement.} ChatGPT 5.6 Sol and 6.0 Astra were used to catch some typos during the journal's final proofing stage; AI was not used for anything else in the document. 

\bibliographystyle{amsplain}

\end{document}